\documentclass[11pt]{article}
\usepackage{graphicx}
\usepackage{bm}
\usepackage{amsmath}
\usepackage{amssymb}
\usepackage{latexsym}
\usepackage{epsfig}
\usepackage{amsbsy}
\usepackage{array}
\usepackage{amssymb}
\usepackage{setspace}
\usepackage{float}
\usepackage{subfigure}
\usepackage[dvips]{color}
\usepackage{mathrsfs}
\usepackage{multirow}
\usepackage{algorithm}
\usepackage{algorithmic}
\usepackage{xcolor}
\usepackage{array}
\usepackage{float}
\usepackage{mathrsfs}
\usepackage{ushort}
\usepackage[misc]{ifsym}
\usepackage{bbding}
\usepackage{anysize}
\usepackage{latexsym}
\usepackage{amsthm}
\usepackage{epsfig}
\usepackage{subfigure}

\marginsize{1in}{1in}{0.5in}{1in}

\newtheorem{theorem}{Theorem}
\newtheorem{corollary}[theorem]{Corollary}
\newtheorem{lemma}[theorem]{Lemma}

\newtheorem{proposition}[theorem]{Proposition}
\newtheorem{definition}[theorem]{Definition}

\newtheorem{assumption}[theorem]{Assumption}

\newtheorem{remark}[theorem]{Remark}

\newcommand{\ubar}[1]{\mkern2mu\underline{\mkern-1.5mu#1\mkern-2.5mu}\mkern4mu }
\def\Is{\int_{\mathbb{S}^2}}
\def \mx {\mathbf{x}}

\def \my {\mathbf{y}}
\def \mz {\mathbf{z}}

\def \mY {\mathbf{Y}}
\def \St {\mathbb{S}^2}
\def \Sd {\mathbb{S}^d}

\def \bR {\mathbb{R}}
\def \Xc {\hat{X}_N}
\def \mw {\mathbf{w}}

\def \me {\mathbf{e}}

\def \als {a_\ell^{(s)}}
\def \aals {\alpha_\ell^{(s)}}
\def \sumlo {\sum_{\ell=1}^{\infty}}
\def \sumk {\sum_{k=1}^{2\ell+1}}

\def \Hsst {\mathbb{H}^s(\St)}

\def\Is{\int_{\mathbb{S}^2}}

\title{Spherical $t_\epsilon$-Designs for Approximations on the Sphere}

\author{Yang Zhou\thanks{Department of Applied Mathematics, The Hong Kong Polytechnic University. (Email: andres.zhou@connect.polyu.hk).
 This author's work is supported in part by Hong Kong Research Council Grant PolyU5002/13p.}
              \and
Xiaojun Chen\thanks{Department of Applied Mathematics, The Hong Kong Polytechnic University
             (Email: maxjchen@polyu.edu.hk).
              This author's work is supported in part by Hong Kong Research Council Grant PolyU5001/12p.}}
\begin{document}

\maketitle

\begin{abstract}
A spherical $t$-design is a set of points on the sphere that are nodes of a positive equal weight quadrature rule  having algebraic accuracy $t$ for all spherical polynomials with degrees  $\le t$.
Spherical $t$-designs have many distinguished properties in approximations on the sphere and receive remarkable attention.
Although the existence of a spherical $t$-design is known for any $t\ge 0$, a spherical design is only known in  a set of interval enclosures on the sphere \cite{chen2011computational} for $t\le 100$.
It is unknown how to choose  a set of points from the set of interval enclosures to obtain a spherical $t$-design.
In this paper we investigate a new concept of point sets on the sphere named spherical $t_\epsilon$-design ($0<\epsilon<1$), which are nodes of a positive weight quadrature rule with algebraic accuracy $t$.
The sum of the weights is equal to the area of the sphere and the mean value of the weights is equal to the weight of the quadrature rule defined by the spherical $t$-design.
A spherical $t_\epsilon$-design is a spherical $t$-design when $\epsilon=0,$ and a spherical $t$-design is a spherical $t_\epsilon$-design for any $0<\epsilon <1$.
We show that any point set chosen from the set of interval enclosures \cite{chen2011computational} is a spherical $t_\epsilon$-design.
We then study the worst-case errors of quadrature rules using spherical $t_\epsilon$-designs in a Sobolev space, and investigate a model of polynomial approximation with the $l_1$-regularization using spherical $t_\epsilon$-designs.
Numerical results illustrate good performance of spherical $t_\epsilon$-designs for numerical integration and function approximation on the sphere.

\end{abstract}

%

{\bf AMS suject classifications:} 65D30, 41A10, 65G30

\pagestyle{myheadings}
\thispagestyle{plain}
\markboth{Y. ZHOU, X. CHEN}{SPHERICAL $t_\epsilon$-DESIGNS FOR  APPROIAMTION ON THE SPHERE}

\section{Introduction}
For a $d$-dimensional sphere
$$\mathbb{S}^d := \{\mx = (x_1,\ldots,x_{d+1})^T \in \mathbb{R}^{d+1} \ | \  \|\mx\|^2_2 =1\},$$
where $\|\cdot\|_2$ means the Euclidean norm, a spherical $t$-design \cite{delsarte1977spherical} for a given positive integer $t$ is a set of $N$ points $X_N = \{ \mx_1,\ldots,\mx_N\}\subset \Sd$ such that
\begin{equation}\label{sphericaldesign}  \frac 1N \sum_{j = 1}^{N}p(\mx_j) = \frac1{|\Sd|} \int_{\Sd}p(\mx){\rm d}\omega_d(\mx)
\end{equation}
holds for all spherical polynomials $p$ with degree $\le t$, in which $|\Sd|$ is the area of $\Sd$ and ${\rm d}\omega_d$ is the surface measure.
For a numerical integration rule on the sphere, we say that the rule has \textit{algebraic accuracy $t$}, or the rule is of \textit{$t$-algebraic accuracy}, if the rule integrates all spherical polynomials with degree $\le t$ exactly. A spherical $t$-design establishes a positive equal weight quadrature rule with algebraic accuracy $t$ for numerical integration on the whole sphere, which is also proved to perform well for numerical integration of spherical  functions belonging to Sobolev spaces, see \cite{brauchart2014qmc} for detail.

The existence of spherical $t$-designs for arbitrary degree $t$ was proved by Seymour et al \cite{seymour1984averaging} in 1984.
Consequently, a natural problem is to find the minimal number of points such that (\ref{sphericaldesign}) holds for fixed $t$ and $d$, denoted as  $N(d, t)$.
In 1993, Korevaar et al \cite{korevaar1993spherical} proved that $N(d, t) \leq C_dt^{(d^2+d)/2}$ and conjectured that $N(d, t) \leq  C_dt^d$ for a sufficiently large positive constant $C_d$ depending only on $d$. This conjecture was then proved  by Bondarenko et al \cite{bondarenko2010optimal} in 2011.
For $d = 2$, there is an even stronger conjecture by Hardin et al \cite{hardin1996mclaren} saying that $N(2, t) \leq \frac12 t^2 + o(t^2)$ as $t \to \infty$. Numerical evidence supporting the conjecture was also given in \cite{hardin1996mclaren,sloan2009variational}.
Spherical $t$-designs have been extensively studied from various viewpoints, among which the application to polynomial approximation and the number of points needed to construct a spherical $t$-design have been paid great attention, see \cite{AN_CHEN,an2012regularized,bajnok1992construction,bannai1979tight,bannai2009survey,bannai2012remarks,chen2011computational,hardin1996mclaren,korevaar1993spherical,sloan2009variational}.
Moreover, numerical methods  have been developed for finding spherical $t$-designs. Usually in those methods  finding spherical designs is reformulated to problems of  solving a system of nonlinear equations or optimization problems, see \cite{AN_CHEN,graf2011computation,sloan2009variational}. Numerical results suggest  that those methods can find approximate spherical $t$-designs with high precision.

However, in general, for any given $t$ the exact location of a spherical $t$-design $X_N=\{\mx_1,\ldots,\mx_N\}$ is unknown. The best we know is that there is a set of points $\hat{\mx}_i, \ i=1,\ldots, N$ such that a  set of narrow intervals defined by
$$
\mathbb{X}_N = \{ [\mx]_i = \mathcal{C}(\hat\mx_i,\gamma_i),\  i = 1,\ldots,N, \ \hat\mx_i \in \St, \ \gamma_i >0\} \subset \St$$
can be computed to contain a spherical $t$-design for $d =2$, $N = (t+1)^2$ and $t\le 100$ in \cite{chen2011computational}, where
$$
\mathcal{C}(\hat\mx_i,\gamma_i) = \{\mx \in \mathbb{S}^2 \mid \cos^{-1}(\mx \cdot \hat\mx_i) \leq \gamma_i \}.
$$

Among all the spherical polynomials with degree $\le t$,  if the zero polynomial is the only one that vanishes at each point in the set $X_N \in \St$, then the point set $X_N \in \St$ is said to be fundamental with order $t$.
In 2011, Chen et al \cite{chen2011computational} proposed a computational-assisted proof for the existence of spherical $t$-designs on $\St$ with $N = (t + 1)^2$ for all values of $t \leq 100$.
An interval arithmetic based algorithm is proposed to compute a series of sets of  polar coordinates type interval enclosures containing fundamental spherical $t$-designs.
By choosing the center points of each interval enclosures, an approximate spherical $t$-design can be obtained and numerical results show that the Weyl sums of these point sets are very close to 0.
However, though we have known the existence of spherical $t$-designs for $t\le 100$ and $N= (t+1)^2$ in a set of interval enclosures, we still can not obtain an exact spherical $t$-design in the set of interval enclosures for a positive equal weight quadrature rule with algebraic accuracy $t$. Motivated by this problem,  in this paper we  relax the equal weights to the ones whose mean value is still $\frac{|\Sd|}{N}$ but can be chosen in an interval with respect to a number $0 \le \epsilon <1$.

\begin{definition}\label{t_epsi}{ \emph{(Spherical $t_\epsilon$-deisgn)}}
	A spherical $t_\epsilon$-design with $ 0 \leq \epsilon < 1$ on $\mathbb{S}^{d}$ is a set of points
	$X^\epsilon_N := \{ \mathbf{x}_1^\epsilon,\ldots,\mathbf{x}_N^\epsilon\} \subset \mathbb{S}^{d} $ such that
	the quadrature rule with  weights $\mw =(w_1,\ldots,w_N)^T$ satisfying
	\begin{equation}\label{a_t_design2}
	\frac{|\Sd|}{N}(1-\epsilon) \leq w_i \leq \frac{|\Sd|}{N}(1-\epsilon)^{-1}, \quad i = 1,\ldots,N,
	\end{equation}
	is exact for all spherical polynomials $p$ of degree at most $t$,  that is,
	\begin{equation}\label{a_t_design}
	\sum_{i=1}^{N}w_i  p(\mathbf{x}^\epsilon_i) = \int_{\Sd} p(\mathbf{x}) {\rm d}\omega_d(\mathbf{x}).
	\end{equation}
\end{definition}

The concept of spherical $t_\epsilon$-designs establishes a bridge of spherical $t$-designs and ordinary positive weight quadrature rules with  algebraic accuracy $t$.

\begin{remark}
	A spherical $t$-design  is  a spherical $t_0$-design with $\epsilon = 0$. By letting $p(\mx) \equiv 1$ in (\ref{a_t_design}) we can obtain $\sum_{i =1}^{N}w_i = |\Sd|$ and thus $ 0 < w_i < |\Sd|$ for $i = 1,\ldots,N$.
\end{remark}

Since the existence of spherical $t$-designs has been proved for arbitrary $t$, and a spherical $t$-design is also a spherical $t_\epsilon$-design for arbitrary $0 \le \epsilon <1$, we have the existence of spherical $t_\epsilon$-designs.
Due to the relaxation of the weights $\mw$, an important advantage is that we can have a positive weight quadrature rule with algebraic accuracy $t$.
Moreover, our numerical experiments show that with the increase of $\epsilon$ we can get numerical integration with  polynomial precision using fewer points than spherical $t$-designs.

The rest of this paper is organized as the following. In Section 2, we discuss the relationship between spherical $t_\epsilon$-designs and spherical $t$-designs when they are fundamental systems with same number of points.
Based on the results  we  study the sets of interval enclosures containing fundamental spherical $t$-designs in Section 3.
We prove that all point sets arbitrarily chosen in those sets of interval enclosures computed in \cite{chen2011computational} are spherical $t_\epsilon$-designs.
In Section 4, we analyze the worst-case errors of spherical $t_\epsilon$-designs for numerical integration on the unit sphere $\St$.
Numerical results show that compared with the equal weight quadrature rules using spherical $t$-designs, the worst-case errors can be improved with relaxing the weights to define quadrature rules  using spherical $t_\epsilon$-designs.
In Section 5, we investigate an $l_2-l_1$ regularized  weighted least squares model for polynomial approximation on the two-sphere using spherical $t_\epsilon$-designs and present  numerical results to demonstrate the efficiency of the $l_2-l_1$ model.

In this paper we concentrate on the case $d =2$.
Throughout the paper we assume that all the points in a point set on the unit sphere are distinct.
The computation is implemented in Matlab 2012b and  done on a Lenovo Thinkcenter PC equipped with Intel Core i7-3770 3.4G Hz CPU, 8 GB RAM running Windows 7.

\section{Spherical $t_\epsilon$-designs: neighborhood of  spherical $t$-designs}

In this section we will study the relationship between spherical $t_\epsilon$-designs and spherical $t$-designs when they are both fundamental systems and have the same number of points.
A spherical $t$-design defines an equal weight quadrature rule with algebraic accuracy $t$ while a spherical $t_\epsilon$-design defines a positive weight quadrature rule with algebraic accuracy $t$.
Based on these two properties, in this section we study a neighborhood of spherical $t$-designs.
Denote
\begin{eqnarray*}
	\mathbb{P}_{t} := \mathbb{P}_{t}(\St) &=& \{ \mbox{spherical \ polynomials \ of \ degree \ $\leq t$ } \} \\
	\quad &=& \mathrm{span} \{Y_{\ell,k}:\ell = 0,\ldots,t,k=1,\ldots,2\ell+1\},
\end{eqnarray*}
as the space of all spherical polynomials with degree $\le t$ on the unit two-sphere $\St$. Here $Y_{\ell,k}$ is a fixed $\mathbb{L}_{2}$-orthogonormal real spherical harmonic of degree $\ell$ and order $k$, which means
\begin{equation}\label{Ylkint}
\int_{\mathbb{S}^2}Y_{\ell,k}(\mx)Y_{\ell^\prime,k^\prime}(\mx){\rm d}\omega(\mathbf{x}) = \delta_{\ell,\ell^\prime}\delta_{k,k^\prime}, \quad \ell,\ell^\prime=0,\ldots,t; \ k,k^\prime=1,\ldots,2\ell+1,
\end{equation}
where ${\rm d}\omega(\mx) = {\rm d}\omega_2(\mx)$, and $\delta_{\ell,\ell'}$ is the Kronecker delta. It is well known that
\begin{equation}\label{dim_P}
d_t := \mathrm{dim}(\mathbb{P}_{t}) = \sum_{\ell = 0}^{t}(2\ell+1)  = (t+1)^2.
\end{equation}
By the addition theorem we have
\begin{equation}
\sum_{k  = 1}^{2\ell +1}  Y_{\ell,k}(\mx)Y_{\ell,k}(\my) = \frac{2\ell +1}{4\pi} P_{\ell}(\mx \cdot \my),
\end{equation}
which implies $ \sum_{k =1}^{2\ell +1} Y_{\ell,k}^2(\mx)  = \frac{2\ell+1}{4\pi}$, where $P_{\ell}, \ \ell \ge 0$ denotes the \textit{Legendre Polynomial} and $\mx \cdot \my $ denotes the Euclidean inner product. Hence we obtain
\begin{equation}\label{Y_lk_upperbound}
\|Y_{\ell,k}\|_{C(\St)} \leq \max_{\mx \in \St}\left(\sum_{k =1}^{2\ell +1} Y_{\ell,k}^2(\mx)\right)^{\frac12} = \sqrt{\frac{2\ell +1}{4\pi}} \quad {\rm for} \  k = 1,\ldots,2\ell+1, \ \ell \ge 0.
\end{equation}
Indeed, a spherical coordinate form of real spherical harmonics  can be represented as (see \cite{abramowitz1972handbook,atkinson2012spherical})
\begin{equation}\label{Ylk}
Y_{\ell,k}(\mx) = \left\{
\begin{array}{ll}
\sqrt{2}N_{\ell,k}P_\ell^{\ell +1 -k}(\cos \theta ) \cos k \varphi , & k = 1 ,\ldots,\ell, \\
\\
N_{\ell,k}P_\ell^{0}(\cos \theta ) , & k = \ell +1,\\
\\
\sqrt{2}N_{\ell,k}P_\ell^{k-\ell-1}(\cos \theta ) \sin k \varphi , & k = \ell +2,\ldots,2\ell+1, \\
\end{array}
\right.
\end{equation}
with
$$
\mx = \left(
        \begin{array}{c}
          \sin \theta \cos \varphi \\
           \sin \theta \sin \varphi\\
            \cos \theta\\
        \end{array}
      \right),
$$
where $ 0 \le\theta \le \pi$, $0\le \varphi < 2\pi$ and $N_{\ell,k}$ are the normalization coefficients
\begin{equation*}
N _{\ell,k} = \sqrt{\frac{2\ell+1}{4\pi} \frac{(\ell-|k-\ell-1|)!}{(\ell+|k-\ell-1|)!}}, \quad k = 1,\ldots,2\ell+1.
\end{equation*}
When taking $k = \ell +1$ and $\theta = 0$ we can obtain $Y_{\ell,\ell+1}(\mx) \equiv \sqrt{\frac{2\ell +1}{4\pi}}$ with $\mx = (0,0,1)^T$.
Therefore, (\ref{Y_lk_upperbound}) is a sharp upper bound of $\|Y_{\ell,k}\|_{C(\St)}$.
With the fact that all spherical harmonics are a basis of $\mathbb{P}_{t}$ we have that
a finite point set $X_N = \{\mx_1,\ldots,\mx_N\}$ is a spherical $t$-design if and only if the \emph{Weyl sums}
\begin{equation}\label{weyl_sum}
\sum_{i = 1}^N Y_{\ell,k}(\mx_i) = 0, \quad k = 1,\ldots,2\ell+1, \ \ell = 1,\ldots,t,
\end{equation}
hold,
see \cite{delsarte1977spherical,sloan2009variational} for details.
Let $X_N = \{ \mx_1,\cdots,\mx_N\} \subset \St$ and $X_N' = \{ \mx_1',\cdots,\mx_N'\}\subset \St$ be two point sets on the sphere.
To describe the relationship among  points and point sets, we introduce the following definitions of distances.
\begin{itemize}
  \item[\textbf{1.}] Define the \emph{geodesic distance} between two points $\mx_i,\mx_j \in \St$ as
$$
{\rm dist}(\mx_i,\mx_j) = \cos^{-1}(\mx_i \cdot \mx_j). $$
  \item[\textbf{2.}] Define the \emph{separation distance} of a point set $X_N$ as
\begin{equation*}
  \rho(X_N) = \min_{i \neq j} \cos^{-1}(\mx_i \cdot \mx_j),
\end{equation*}
which represents the minimal geodesic distance between two different points  in $X_N$.
  \item[\textbf{3.}] Define the \emph{least distance} from a point $\mx \in \St$  to a point set $X_N $ as
$${\rm dist}(\mx , X_N) = \min_{ i } \ {\rm dist}(\mx,\mx_i) = \min_{i } \ \cos^{-1}(\mx \cdot \mx_i),$$
which can  be seen as the geodesic distance from $\mx$ to its projection  on $X_N$.
  \item[\textbf{4.}] Define the \textit{Hausdorff distance} between two point sets  $X_N$ and $X_N'$ by
\begin{eqnarray}\label{hausdorff}
 \nonumber \sigma(X_N,X_N') &=& \max\{\max_{  i} {\rm dist}(\mx_i^\prime , X_N) , \ \max_{ i } {\rm dist}(\mx_i , X_N')\} \\
   &=&  \max\{\max_i\min_j\cos^{-1}(\mx_i^\prime \cdot \mx_j),\max_j\min_i\cos^{-1}(\mx_i^\prime \cdot \mx_j)\}.
\end{eqnarray}
\end{itemize}
Note that $  \sigma(X_N,X_N')  =  \sigma(X_N',X_N)$ and $ \sigma(X_N,X_N') = 0$ if and only if $X_N = X_N'$.


\begin{remark}\label{hausdorff_remark}
For two point sets $X_N$ and $X_N'$, if $\sigma(X_N,X_N')<\frac12\rho(X_N)$,
 then for each $\mx_i \in X_N$ there exists a unique $\mx_j' \in X_N'  \cap \mathcal{C}(\mx_i,\frac12\rho(X_N))$, where
$$
\mathcal{C}(\mx_i,\frac12\rho(X_N)) = \{\mx \in \mathbb{S}^2 \mid \cos^{-1}(\mx \cdot \mx_i) \leq \frac12\rho(X_N) \}.
$$
On account of the relationship between $\mathcal{C}(\mx_i,\frac12\rho(X_N))$ and  $X_N'$, in what follows we will denote $\mx_i'$ as the point belonging to $\mathcal{C}(\mx_i,\frac12\rho(X_N))$ for sake of consistency under the assumption $\sigma(X_N,X_N')<\frac12\rho(X_N)$.
\end{remark}

For a point set $X_N = \{ \mx_1,\cdots,\mx_N\}\subset \St$ we define the matrix $ \mathbf{Y}(X_N) \in \mathbb{R}^{N\times d_{t}}$ with its elements as
\begin{equation}\label{mY}
  \mathbf{Y}_{i,\ell^2+ k}(X_N) =Y_{\ell,k}(\textbf{x}_{i}), \quad i=1,\ldots,N,\ k=1,\ldots,2\ell+1,\ \ell=0,\ldots,t.
\end{equation}
Note that $|\St | = 4\pi$.
An equivalent condition of spherical $t_\epsilon$-designs given in \cite{annual14} states that a point set $X^\epsilon_N := \{ \mathbf{x}^\epsilon_1,\ldots,\mathbf{x}^\epsilon_N\} \subset \mathbb{S}^2 $ is a spherical $t_\epsilon$-design if and only if
\begin{equation}\label{Yw}
	\mathbf{Y}(X^\epsilon_N)^T \mw - \sqrt{4\pi}  \me_1 = \mathbf{0}
	~~~~{\rm and} \quad   \frac{4\pi(1-\epsilon)}{N}\me \le \mw \le \frac{4\pi(1-\epsilon)^{-1}}{N}\me,
\end{equation}
where
$\me_1 = (1,0,\ldots,0)^T \in \bR^{(L+1)^2}$ and $\me = (1,\ldots,1)^T \in \mathbb{R}^{N}$.

For two matrices constructed by two near enough point sets on $\St$ we have the following property.

\begin{proposition}\label{proposition2.2}
	For any two point sets $X_N$, $X_N' \subset \St$ satisfying $\sigma(X_N,X_N') < \frac12 \rho(X_N)$ there always holds
	\begin{eqnarray}\label{difference_matrixnorm_bound}
	\nonumber \|( \mY(X_N) - \mathbf{Y}(X_N'))^T\|_\infty & = & \| \mY(X_N) - \mathbf{Y}(X_N')\|_1 \\
	& \leq & N(t+1)\sqrt{\frac{2t+1}{4\pi}}\sigma(X_N,X_N'),
	\end{eqnarray}
	where $\mY(X_N), \mY(X_N') \in \bR^{N\times d_t}$ are matrices defined by (\ref{mY}) which depend on $t$.
\end{proposition}

\begin{proof}
For a point $\mx_i \in X_N$, by Remark \ref{hausdorff_remark} we let $\mx_i'$ be the unique point  located in $ \mathcal{C}(\mx_i,\frac12 \rho(X_N)) \cap X_N'$. Let $Q_{\ell,k}$ be the restriction  of $Y_{\ell,k}$ on the great circle through these two points. Then $Q_{\ell,k}$ is a trigonometric polynomial on the sphere and by \textit{Bernstein's inequality} \cite{borwein1995polynomials} and (\ref{Y_lk_upperbound}) we obtain
	\begin{eqnarray}\label{Ylkdifference}
	\nonumber |Y_{\ell,k}(\mx_i)-Y_{\ell,k}(\mx_i')|  & = & |Q_{\ell,k}(\mx_i)-Q_{\ell,k}(\mx_i')| \\
	\nonumber  & \leq & \cos^{-1}(\mx_i\cdot\mx_i') \sup|Q_{\ell,k}'|\\
	\nonumber  & \leq & \cos^{-1}(\mx_i\cdot\mx_i') (t+1) \sup|Q_{\ell,k}| \\
	\nonumber & \leq & \cos^{-1}(\mx_i\cdot\mx_i') (t+1) \|Y_{\ell,k}\|_{C(\St)}\\
   & \le & \sigma(X_N,X_N')(t+1)\sqrt{\frac{2\ell+1}{4\pi}},
	\end{eqnarray}
	where the last inequality is obtained by (\ref{Y_lk_upperbound}).
Together with (\ref{difference_matrixnorm_bound}) and (\ref{Ylkdifference})  we have
\begin{eqnarray*}
  \| \mY(X_N) - \mathbf{Y}(X_N')\|_1 &=& \max_{ 0 \le \ell \le t , 1 \le k \le 2\ell+1} \sum_{j =1}^{N} |Y_{\ell,k}(\mx_j)-Y_{\ell,k}(\mx_j')| \\
   & \le & N(t+1)\sqrt{\frac{2t+1}{4\pi}}\sigma(X_N,X_N').
\end{eqnarray*}
\end{proof}

Let $X_N^0 = \{\mx_1^0,\ldots,\mx_N^0\} \subset \St $ be a fundamental spherical $t$-design. Given a number $\sigma^\ast \ge 0$, denote the neighborhood of $X_N^0$ with radius $\sigma^\ast$ by
$$
\mathcal{C}(X_N^0,\sigma^\ast) = \big\{ X_N \subset \St : \sigma(X_N,X_N^0) \le \sigma^\ast
\big\}.
$$
The following lemma indicates that any point set contained in a small enough neighborhood of a fundamental spherical $t$-design is a fundamental spherical $t_\epsilon$-design.

\begin{lemma} \label{sec1theo}
Let  $X_N^0$ be a fundamental spherical $t$-design with order $t$ and $N = (t+1)^2$. Then any point set $X_N \in \mathcal{C}(X_N^0,\sigma^\ast)$ is a fundamental spherical $t_\epsilon$-design with
\begin{equation}\label{epsi_bound1}
\frac{\tau\sigma^\ast\|\mY(X_N^0)^{-1}\|_1 }{1-\tau\sigma^\ast\|\mY(X_N^0)^{-1}\|_1 }  \le \epsilon <1,
\end{equation}
where
\begin{equation}\label{sigma_star}
\sigma^\ast < \frac12 \min \left( \frac1{\tau\|\mY(X_N^0)^{-1}\|_1} , \rho(X_N^0) \right),
\end{equation}
 with $ \tau = \sqrt{\frac{2t+1}{4\pi}}(t+1)^3$.
\end{lemma}

\begin{proof}
This lemma can be proved by showing that for any  point set $X_N \in \mathcal{C}(X_N^0,\sigma^\ast)$ we have
\begin{equation}\label{at_design_3}
\mY(X_N)^T \mw = \sqrt{4 \pi} \me_1,
\end{equation}
in which $\mY(X_N)$ is nonsingular and $\| \mw - \frac{4\pi}{N} \me\|_\infty < \frac{4\pi}{N}$.

From Proposition \ref{proposition2.2} for any point set $X_N$ with $\sigma(X_N,X_N^0)<\sigma^\ast$  we have
$$\|\mY(X_N^0)^{-1}\|_1 \|\mY(X_N^0) - \mathbf{Y}(X_N)\|_1 <\|\mY(X_N^0)^{-1}\|_1\tau\sigma^{\ast} < 1.$$
Hence $\mY(X_N)$ is nonsingular, that is, $X_N$ is a fundamental system with order $t$.

By the fact that $X_N^0$ is a fundamental spherical $t$-design and (\ref{weyl_sum}) we have
\begin{equation}\label{sph_design}
	\frac{4\pi}{N} \mathbf{Y}(X_N^0)^T \me = \sqrt{4 \pi} \me_1.
\end{equation}
Together with the well known perturbation theorem of linear systems in \cite[Theorem 2.3.9, pp.135]{watkins2004fundamentals} and
$$ \| \mathbf{I} - \mY(X_N^0))^{-1}\mY(X_N)^T\|_\infty \leq \|\mY(X_N^0)^{-1}\|_1\| \mY(X_N^0) -\mY(X_N)\|_1 <1, $$ we have
\begin{equation}\label{inequa_2}
  \frac{\|\mw-\frac{4\pi}{N}\me\|_{\infty}}{\|\frac{4\pi}{N}\me\|_{\infty}} \leq \frac{\|(\mY(X_N^0)^T)^{-1}\|_{\infty} \ \|(\mY(X_N) - \mathbf{Y}(X_N^0))^T\|_{\infty}}{1-\|(\mY(X_N^0)^T)^{-1}\|_\infty\|(\mY(X_N) - \mathbf{Y}(X_N^0))^T\|_{\infty} }.
\end{equation}

Therefore, we obtain
\begin{eqnarray*}
   \|\mw-\frac{4\pi}{N}\me\|_\infty &=&\frac{4\pi}{N}\frac{\|\mw-\frac{4\pi}{N}\me\|_\infty}{\|\frac{4\pi}{N}\me\|_\infty} \\
   &\leq &  \frac{4\pi}{N} \frac{\|\mY(X_N^0)^{-1}\|_1 \|\mY(X_N^0) - \mathbf{Y}(X_N)\|_1}{1-\|\mY(X_N^0) - \mathbf{Y}(X_N)\|_1\|\mY(X_N^0)^{-1}\|_1} \\
   & \le &   \frac{4\pi}{N}\frac{\tau\sigma^\ast\|\mY(X_N^0)^{-1}\|_1 }{1-\tau\sigma^\ast\|\mY(X_N^0)^{-1}\|_1 } < \frac{4\pi}{N}. \\
\end{eqnarray*}
Hence the vector $\mw$ is positive which implies that $X_N$ is spherical $t_\epsilon$-design with
$\epsilon$ satisfying (\ref{epsi_bound1}).
\end{proof}

Based on above lemma we can deduce the following result which describes an upper bound of the radius of the neighborhood of a fundamental spherical $t$-design, in which any point set located in the neighborhood is a fundamental spherical $t_\epsilon$-design.

\begin{corollary}\label{sec1coro}
Let  $X_N^0$  be a  fundamental spherical $t$-design with order $t$ and $N = (t+1)^2$.
For any $0 \le \epsilon <1$, if
\begin{equation}\label{corollary2_4}
 \sigma(X_N,X_N^0) < \min \left( \frac12\rho(X_N^0), \frac{\epsilon}{\tau(1+\epsilon)\|\mY(X_N^0)^{-1}\|_1 }\right),
\end{equation}
then $X_N$ is a fundamental spherical $t_\epsilon$-design.
\end{corollary}

\begin{proof}
By the fact that $\sigma(X_N,X_N^0) < \sigma^\ast$ which is defined in (\ref{sigma_star}), we have $\mY(X_N)$ is nonsingular and
\begin{eqnarray*}
   \|\mw-\frac{4\pi}{N}\me\|_\infty  \leq    \frac{4\pi}{N}\frac{\tau\sigma(X_N,X_N^0)\|\mY(X_N^0)^{-1}\|_1 }{1- \tau\sigma(X_N,X_N^0)\|\mY(X_N^0)^{-1}\|_1} <  \frac{4\pi}{N}\epsilon.
\end{eqnarray*}
Hence, from (\ref{corollary2_4}), we derive
$$ ( 1 - \epsilon )  \frac{4\pi}{N }  \me < \mw < (1 + {\epsilon})\frac{4\pi}{N }  \me  \le \frac{4\pi (1-\epsilon)^{-1}}{N }  \me .$$
We complete the proof.
\end{proof}

\section{Interval analysis of spherical $t$-designs and spherical $t_\epsilon$-designs}
In this section we will study the sets of interval enclosures containing fundamental spherical $t$-designs.
In the last section we describe a neighborhood of a fundamental spherical $t$-design in which any point set is a fundamental spherical $t_\epsilon$-design.
A series of sets of interval enclosures of spherical $t$-designs can be computed in \cite{chen2011computational}, but the exact location of spherical $t$-designs can not be obtained.
In the following we will show that any point set in the set of the interval enclosures given in \cite{chen2011computational} is a fundamental spherical $t_\epsilon$-design.
Let
\begin{equation}\label{setX}
\mathbb{X}_N = \{[\mx]_i = \mathcal{C}(\hat{\mx}_{i},\gamma_i) \subset \St , \ i = 1 , \ldots, N\}
\end{equation}
be a set of spherical caps, with $\hat{\mx}_i$ as the center point and $\gamma_i$ as the radius.
Additionally, let $\hat{X}_N =\{ \hat{\mx}_{1}, \ldots , \hat{\mx}_{N}\}  \subset \St$
be the set of center points.
Define the radius of $\mathbb{X}_N$ by $${\rm rad} (\mathbb{X}_N) =  \max_{1 \leq i \leq N} \gamma_i,$$
and the separation distance of $\mathbb{X}_N$ by
$$
\rho(\mathbb{X}_N) =  \min_{ \tiny {\begin{array}{c}
                               i \neq j \\
                              \mx_i \in [\mx]_i,   \mx_j \in [\mx]_j,
                            \end{array}   }
}{\rm dist}(\mx_i , \mx_j).$$
We say that $\mathbb{X}_N$ is an interval enclosure of a point set $X_N=\{\mx_1,\ldots,\mx_N\}$, denoted as $X_N \in \mathbb{X}_N$, if $\mx_i \in [\mx]_i$ and $\mx_i \notin [\mx]_j$  for $ i = 1 ,\ldots,N$ and $ i \neq j$.

\begin{assumption}\label{Assumption2.1}
Let $\mathbb{X}_N$ defined by (\ref{setX}) be a set of intervals. Assume that
\begin{enumerate}
 \item there exists  a  spherical $t$-design $X_N^0 \in \mathbb{X}_N$;
  \item $\mY(\hat{X}_N)$ is nonsingular.
\end{enumerate}
\end{assumption}

The assumption that $\mY(\hat{X}_N)$ is nonsingular also implies that $\mY(\Xc)$ is a square matrix which requires $N = (t+1)^2$.
In the following theorem we show that under Assumption \ref{Assumption2.1} if ${\rm rad}(\mathbb{X}_N)$ is smaller than a certain number, then $\mY(X_N)$ is nonsingular and (\ref{corollary2_4}) holds for any $X_N \in \mathbb{X}_N$.

\begin{theorem}\label{maintheochap12}
Under Assumption \ref{Assumption2.1},
any point set $X_N \in \mathbb{X}_N$  is a fundamental spherical $t_\epsilon$-design
with
\begin{equation}\label{maintheochap2}
     \frac{2 \tau{\rm rad}(\mathbb{X}_N) \|\mY(\Xc)^{-1}\|_1 }{1 - 4 \tau {\rm rad}(\mathbb{X}_N)\|\mY(\Xc)^{-1}\|_1} \le \epsilon <1,
\end{equation}
if
\begin{equation}\label{theorem3.1}
{\rm rad}(\mathbb{X}_N) < \min \left(\frac14\rho(\mathbb{X}_N), \dfrac{\epsilon}{2(1+2\epsilon)\tau\|\mY(\hat X_N)^{-1}\|_1}\right).
\end{equation}
\end{theorem}
 \begin{proof}
For any $X_N \in \mathbb{X}_N$ it can be concluded that
$$\rho(X_N) \ge \rho(\mathbb{X}_N). $$
Hence for any two point sets $X_N, X_N' \in \mathbb{X}_N$ we have
$$\sigma (X_N,X_N') \leq  \max_{ 1 \leq i \leq  N} \max_{ \tiny{\mx_i,\my_{i} \in [\mx]_i}} {\rm dist}(\mx_i , \my_{i})  = 2 {\rm rad} (\mathbb{X}_N).$$
First we show that $\mY(X_N)$ is nonsingular for any $X_N \in \mathbb{X}_N$. By Proposition \ref{proposition2.2} we have
\begin{eqnarray*}
  \|{\rm I} - \mY(\hat X _N)^{-1} \mY(X_N)\|_1 & \le & \| \mY(\hat X _N)^{-1}\|_1 \|  \mY(\hat X _N)- \mY(X_N)\|_1 \\
   & \le & \| \mY(\hat X _N)^{-1}\|_1\tau \sigma(\hat X_N,X_N) \\
   &\le & 2 \tau \| \mY(\hat X _N)^{-1}\|_1 {\rm rad}(\mathbb{X_N}) <1.
\end{eqnarray*}
Hence all the point sets $X_N \in \mathbb{X}_N$ including $X_N^0$ are fundamental systems with order $t$. Moreover, from (\ref{theorem3.1}) we have
$$\sigma (X_N,X_N^0)   < \frac12 \rho(X_N^0).$$
By (\ref{theorem3.1}) we can also have
$$2(1+\epsilon) \tau {\rm rad}(\mathbb{X}_N) \|\mY(\hat X_N)^{-1}\|_1 < \epsilon - 2 \epsilon   \tau{\rm rad}(\mathbb{X}_N) \|\mY(\hat X_N)^{-1}\|_1.$$
By Corollary 2.7 in \cite[pp.119]{Sun1990matrix} it can be concluded that for arbitrary $X_N \in \mathbb{X}_N$ we have
\begin{equation}\label{corollary2_7}
\|\mY(X_N)^{-1}\|_{1}
\leq \dfrac{ \|\mY(\hat X_N)^{-1}\|_1  } {  1 - \|\mY(\hat X_N)^{-1}\|_1
	\|\mY(\hat{X}_N)-\mY(X_N)\|_{1} }.
\end{equation}
Then together with Proposition \ref{proposition2.2}  we have
\begin{eqnarray}
\nonumber\sigma(X_N,X_N^0) & \le & 2 {\rm rad}(\mathbb{X}_N)\\
\nonumber& < & \frac{\epsilon(1 -  2\tau{\rm rad}(\mathbb{X}_N)\|\mY(\hat X_N)^{-1}\|_1)}{(1+\epsilon)  \tau \|\mY(\hat X_N)^{-1}\|_1} \\
\nonumber& \le & \frac{\epsilon(1 -  \|\mY(\hat X_N)^{-1}\|_1 \|\mY(\hat X_N) - \mY(X_N^0)\|_1)}{(1+\epsilon) \tau \|\mY(\hat X_N)^{-1}\|_1}\\
& \le & \frac{\epsilon}{(1+\epsilon) \tau \|\mY(X_N^0)^{-1}\|_1}.
\end{eqnarray}
By Corollary \ref{sec1coro} we have that any point set $X_N \in \mathbb{X}_N$ is a fundamental spherical $t_\epsilon$-design.
Additionally, together with
\begin{eqnarray*}
	\|\mw-\frac{4\pi}{N}\me\|_\infty & \leq &  \frac{4\pi}{N} \frac{\|\mY(X_N^0)^{-1}\|_1 \|\mY(X_N^0) - \mathbf{Y}(X_N)\|_1}{1-\|\mY(X_N^0) - \mathbf{Y}(X_N)\|_1\|\mY(X_N^0)^{-1}\|_1} \\
	& \le &   \frac{4\pi}{N}\frac{2\tau{\rm rad}(\mathbb{X}_N)\|\mY(X_N^0)^{-1}\|_1 }{1-2\tau{\rm rad}(\mathbb{X}_N)\|\mY(X_N^0)^{-1}\|_1 },
\end{eqnarray*}
and (\ref{corollary2_7}) we have
\begin{eqnarray*}
& & \frac{2\tau{\rm rad}(\mathbb{X}_N)\|\mY(X_N^0)^{-1}\|_1 }{1-2\tau{\rm rad}(\mathbb{X}_N)\|\mY(X_N^0)^{-1}\|_1 }\\
& \le & \frac{2\tau{\rm rad}(\mathbb{X}_N)\dfrac{ \|\mY(\hat X_N)^{-1}\|_1  } {  1 - \|\mY(\hat X_N)^{-1}\|_1
		\|\mY(\hat{X}_N)-\mY(X_N)\|_{1} } }{1-2\tau{\rm rad}(\mathbb{X}_N)\dfrac{ \|\mY(\hat X_N)^{-1}\|_1  } {  1 - \|\mY(\hat X_N)^{-1}\|_1
		\|\mY(\hat{X}_N)-\mY(X_N)\|_{1} } }\\
& \le & \frac{2\tau{\rm rad}(\mathbb{X}_N)\|\mY(\hat X_N)^{-1}\|_1 }{1-4\tau{\rm rad}(\mathbb{X}_N)\|\mY(\hat X_N)^{-1}\|_1 }.
\end{eqnarray*}
Then we have that any point set $X_N \in \mathbb{X}_N$ is a spherical $t_\epsilon$-design with
\begin{equation*}
\frac{2 \tau{\rm rad}(\mathbb{X}_N) \|\mY(\Xc)^{-1}\|_1 }{1 - 4 \tau {\rm rad}(\mathbb{X}_N)\|\mY(\Xc)^{-1}\|_1} \le \epsilon <1,
\end{equation*}
under the  assumption of ${\rm rad}(\mathbb{X}_N)$.
\end{proof}

Theorem \ref{maintheochap12} proves that an arbitrarily chosen point set in a set of interval enclosures of a fundamental spherical $t$-design is a spherical $t_\epsilon$-design if ${\rm rad}(\mathbb{X}_N)$ is smaller than a certain number.
In the theorem we discuss the interval enclosures defined by spherical caps as $[\mx]_i = \mathcal{C}(\hat\mx_i,\gamma_i) = \{ \mx\in \St | \cos^{-1}(\mx\cdot\hat\mx_i)\le \gamma_i\}$ for which many nice properties of real spherical harmonics can be adopted.
However, in practice, to reduce the spherical constraint of points and the dimension of variables, the spherical coordinate form of the points are preferable to compute the interval enclosures, see \cite{chen2011computational}.
For a point $\mx_i \in X_N \subset \St$, denote $\theta_i$, $\varphi_i$ as its spherical coordinate.
Then in \cite{chen2011computational} a series of intervals $[\theta]_i  = [\ubar{\theta}_i,\bar{\theta}_i]$, $ [\varphi]_i = [\ubar{\varphi}_i,\bar{\varphi}_i]$  are computed such that $\mathbb{Z}_N= \{[\mz]_1,\ldots,[\mz]_N\}$ is a set of  interval enclosures of a \textit{well-conditioned spherical $t$-design} \cite{AN_CHEN}, in which each element in $\mathbb{Z}_N$ is defined by
 \begin{equation}\label{Z}
 [\mz]_i = \left(
             \begin{array}{c}
               \sin([\theta]_i)\cos([\varphi]_i) \\
               \\
               \sin([\theta]_i)\sin([\varphi]_i) \\
               \\
               \cos([\theta]_i) \\
             \end{array}
           \right), \quad i = 1,\ldots,N.
\end{equation}
In this sense, different from the interval enclosures defined by the spherical caps, each interval enclosure computed in \cite{chen2011computational} is a rectangle as $[\theta]_i \times [\varphi]_i$.
Therefore, there remains a gap between real computation of interval enclosures of spherical $t$-designs and our analysis above.
Naturally, a strategy to overcome this gap is that for each spherical rectangle in \cite{chen2011computational} we construct a spherical cap which is as small as possible to cover it.
For the spherical rectangle $[\theta]_i \times [\varphi]_i = [\ubar{\theta}_i,\bar{\theta}_i] \times [\ubar{\varphi}_i,\bar{\varphi}_i]$ its four vertices can be written as
\begin{align*}
 \mx_{i,1} = \left(
             \begin{array}{c}
               \sin(\ubar{\theta}_i)\cos(\ubar{\varphi}_i) \\\\
               \sin(\ubar{\theta}_i)\sin(\ubar{\varphi}_i) \\\\
               \cos(\ubar{\theta}_i) \\
             \end{array}
           \right), \quad
                      \mx_{i,2} = \left(
             \begin{array}{c}
               \sin(\ubar{\theta}_i)\cos(\bar{\varphi}_i) \\\\
               \sin(\ubar{\theta}_i)\sin(\bar{\varphi}_i) \\\\
               \cos(\ubar{\theta}_i) \\
             \end{array}
           \right),
           \\\\
            \mx_{i,3} = \left(
             \begin{array}{c}
               \sin(\bar{\theta}_i)\cos(\bar{\varphi}_i) \\\\
               \sin(\bar{\theta}_i)\sin(\bar{\varphi}_i) \\\\
               \cos(\bar{\theta}_i) \\
             \end{array}
           \right), \quad
                      \mx_{i,4} = \left(
             \begin{array}{c}
               \sin(\bar{\theta}_i)\cos(\ubar{\varphi}_i) \\\\
               \sin(\bar{\theta}_i)\sin(\ubar{\varphi}_i) \\\\
               \cos(\bar{\theta}_i) \\
             \end{array}
           \right).
\end{align*}
It can be shown that there exists a point $\hat{\mx}_i$ defined by $ [\theta , \varphi] \in [\theta]_i \times [\varphi]_i$ satisfying
 \begin{equation}\label{equal_distance}
   {\rm dist}(\hat{\mx}_{i}, \mx_{i,j}) = {\rm dist}(\hat{\mx}_i, \mx_{i,k}) \quad {\rm for } \ j, k = 1,2,3,4
 \end{equation}
and $ [\mz]_i \subseteq \mathcal{C}(\hat{\mx}_i,\gamma_i)$ with $\gamma_i = {\rm dist}(\hat{\mx}_i,\mx_{i,1} )$.
However, computing such point $\hat \mx_i$ is
time-consuming and imports large round-off errors when the radii of interval enclosures are small.
Instead of computing $\hat\mx_i$, we investigate another strategy to compute the spherical caps to cover  spherical rectangles which is coarser but more practical.
For a spherical coordinate interval $[\mz]_i$, we use the center point  of the interval $ [\ubar{\theta}_i,\bar{\theta}_i] \times [\ubar{\varphi}_i,\bar{\varphi}_i]$ to define a point  as
\begin{equation}\label{xcs}
 \tilde{\mx}_i = \left(
             \begin{array}{c}
               \sin(\frac12(\bar{\theta}_i+\ubar{\theta}_i))\cos(\frac12(\bar{\varphi}_i+\ubar{\varphi}_i)) \\\\
               \sin(\frac12(\bar{\theta}_i+\ubar{\theta}_i))\sin(\frac12(\bar{\varphi}_i+\ubar{\varphi}_i)) \\\\
               \cos(\frac12(\bar{\theta}_i+\ubar{\theta}_i)) \\
             \end{array}
           \right).
\end{equation}
 Note that the spherical coordinate of $\tilde{\mx}_i$ is the center point of the interval $[\theta]_i \times [\varphi]_i$ but itself is not necessary to be the center point of $[\mz]_i$ in the form of the spherical coordinate. Still, we have that
\begin{equation}\label{xcs_distance}
  {\rm dist}(\tilde{\mx}_i,\mx_{i,1}) =  {\rm dist}(\tilde{\mx}_i,\mx_{i,2}),\quad  {\rm dist}(\tilde{\mx}_i,\mx_{i,3}) =  {\rm dist}(\tilde{\mx}_i,\mx_{i,4}).
\end{equation}
It is obvious to obtain that the distance between $\tilde{\mx}_i$ and any point in $[\mz]_i$ does not exceed the maximum of the four distances in (\ref{xcs_distance}). Therefore, if we let
\begin{equation}\label{gamma_cs}
  \gamma_i = \max\{ {\rm dist}(\tilde{\mx}_i,\mx_{i,1})\ ,\  {\rm dist}(\tilde{\mx}_i,\mx_{i,3})\},
\end{equation}
then we  have
\begin{equation}\label{z_cs}
[\mz]_i \subseteq \mathcal{C}(\tilde{\mx}_i,\gamma_i).
\end{equation}
Consequently, we call the set of spherical caps
$$
\tilde{\mathbb{X}}_N = \{ \mathcal{C}(\tilde{\mx}_i,\gamma_i) \}$$
as a \emph{cap-cover of $\mathbb{Z}_N$}
with $\tilde\mx_i$, $\gamma_i$ defined in (\ref{xcs}) and (\ref{gamma_cs}).
Similar with set of spherical caps $\mathbb{X}_N$, we define the radius and separation distance of $\mathbb{Z}_N$ by
\begin{equation}\label{radius_z}
{\rm rad}(\mathbb{Z}_N) = \max_{1\leq i\leq N} \big\{ \max \{ {\rm dist}(\tilde{\mx}_i,\mx_{i,1})\ ,\  {\rm dist}(\tilde{\mx}_i,\mx_{i,3}) \}\big\} = {\rm rad}(\tilde{\mathbb{X}}_N),
\end{equation}
and
\begin{equation}\label{z_rho}
  \rho(\mathbb{Z}_N) = \min_{\tiny {\begin{array}{c}
                               i \neq j \\
                              1 \leq i,j \leq N
                            \end{array}   }} \{ {\rm dist}(\tilde{\mx}_i,\tilde{\mx}_j)- \gamma_i-\gamma_j \} = \rho(\tilde{\mathbb{X}}_N).
\end{equation}
And then we have the following corollary for the lower bound of $\epsilon$ for $\mathbb{Z}_N$.

\begin{corollary}\label{2nd_corollary}
Let $\mathbb{Z}_N$ be a set of spherical rectangle and its cap-cover $\tilde{\mathbb{X}}_N$ satisfy Assumption \ref{Assumption2.1}.
Then any point set $X_N \in \mathbb{Z}_N$  is a fundamental spherical $t_\epsilon$-design
with
\begin{equation}\label{lowerbound_rectangle}
\frac{2 \tau{\rm rad}(\mathbb{X}_N) \|\mY(\tilde X_N)^{-1}\|_1 }{1 - 4 \tau {\rm rad}(\mathbb{X}_N)\|\mY(\tilde X_N)^{-1}\|_1} \le \epsilon <1,
\end{equation}
if
\begin{equation}
{\rm rad}(\mathbb{Z}_N) < \min \left(\frac14\rho(\mathbb{Z}_N), \dfrac{\epsilon}{2(1+2\epsilon)\tau\|\mY(\tilde X_N)^{-1}\|_1}\right).
\end{equation}
\end{corollary}

The proof of the corollary follows the same manner of Theorem \ref{maintheochap12}.

Based on Corollary \ref{2nd_corollary}, the lower bounds of $\epsilon$, denoted by $\ubar{\epsilon}$, for the sets of interval enclosures provided in \cite{chen2011computational} can be computed.
The data containing the sets of interval enclosures for the parameterization of the spherical $t$-designs and relative programs can be downloaded from the website\\
http://www-ai.math.uni-wuppertal.de/SciComp/SphericalTDesigns.
The computational results are shown in Fig 3.1 and Table 3.1.

\begin{figure}[!h]\label{epsifig}
  \centering
  \includegraphics[width=3.5in]{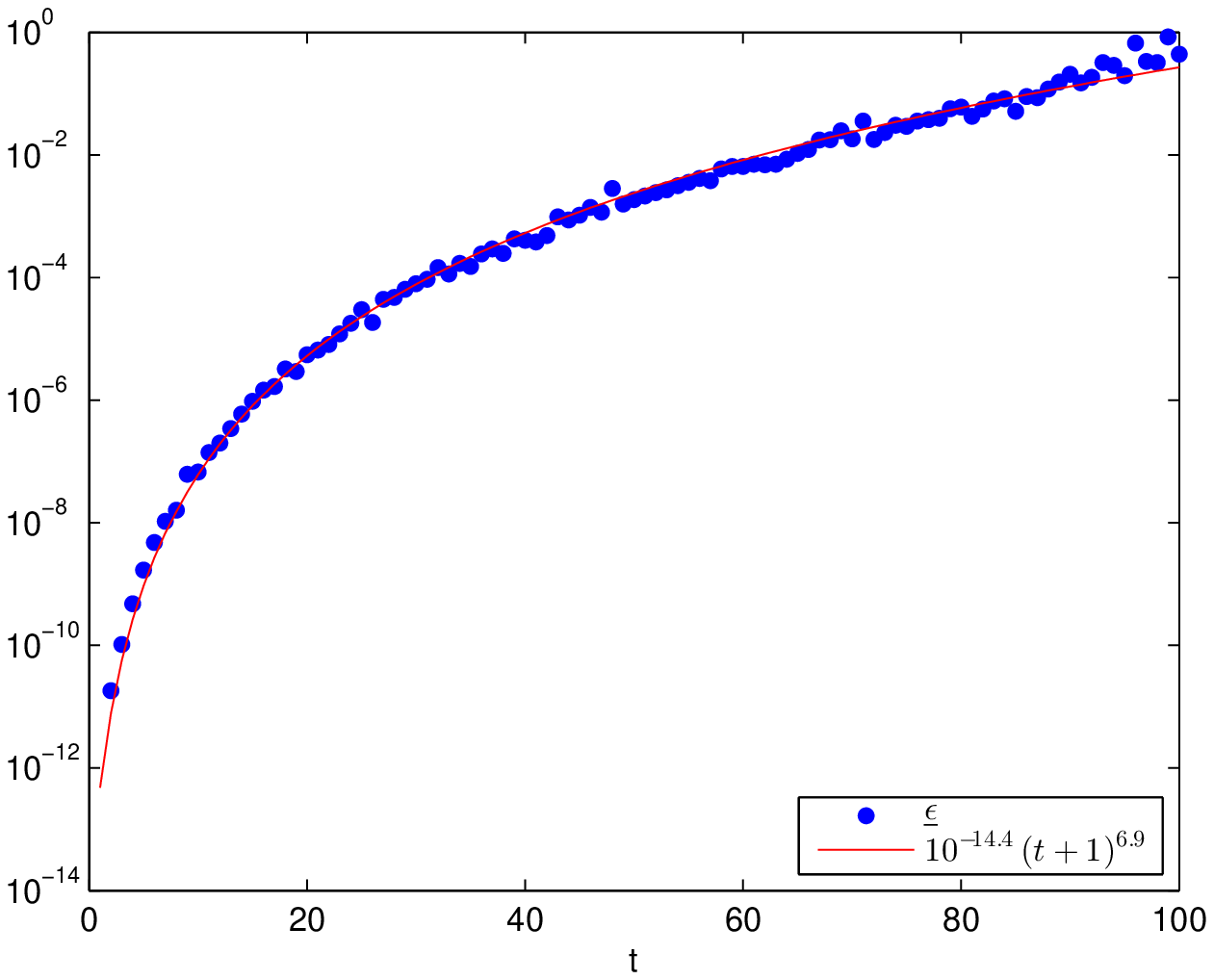}\\
  \caption{$\ubar{\epsilon}$ for $t = 2,\ldots,100$}\label{plotallepsi}
\end{figure}

In Fig. \ref{epsifig}.1, we report the lower bounds of $\epsilon$, denoted as $\ubar{\epsilon}$, for sets of interval enclosures computed in \cite{chen2011computational} for $t = 2,\ldots,100$ (for $t =1$ we have known that the regular tetrahedron is a spherical $t$-design so that $\epsilon = 0$ and we would not consider this case here), based on formula (\ref{lowerbound_rectangle}). In this figure we also plot a function
  \begin{equation}\label{approx_epsi_function}
   y = 10 ^{-14.4}(t+1)^{6.9},
\end{equation}
to approximately estimate the track of $\ubar{\epsilon}$ with respect to $t$.
From the figure we can conclude that the  lower bound of $\epsilon$ grows with  the increase of $t$ in an order about $6.9$. Additionally, by (\ref{lowerbound_rectangle}) it is known that the norm of $\mY(\tilde X_N)^{-1}$ is also very important in the process of estimating $\ubar{\epsilon}$. Fortunately, since the sets of interval enclosures computed in \cite{chen2011computational} seek to include well-conditioned spherical $t$-designs \cite{AN_CHEN}, the growth of lower bounds of $\epsilon$ keeps stable for all the $t$ considered here.

\begin{table}[H]\label{epsitable}\begin{center}\caption{Information for sets of interval enclosures $\mathbb{Z}_N$ for some selected $t$}
\renewcommand\arraystretch{1.1}
\addtolength{\tabcolsep}{3pt}
\begin{tabular}{c|c|c|c|c}
\hline\hline\noalign{\smallskip}
  $t$ & rad($\mathbb{Z}_N$) & $\rho(\mathbb{Z}_N)$ & $\ubar{\epsilon}$ & $\epsilon$ for $\tilde{X}_N$\\
  \hline
10 & 1.843454e-12 & 3.396362e-01 & 6.748890e-08 & 6.694645e-14 \\
20 & 1.515848e-11 & 1.805783e-01 & 5.480524e-06  & 1.783018e-13 \\
30 & 5.588085e-11 & 1.249714e-01 & 7.888659e-05  & 2.480238e-13 \\
40 & 1.044163e-10 & 9.203055e-02 & 4.043164e-04  & 5.339063e-13 \\
50 & 2.199182e-10 & 7.638945e-02 & 1.862348e-03  & 5.057066e-13 \\
60 & 4.006638e-10 & 6.302748e-02 & 6.502352e-03  & 6.747935e-13 \\
70 & 6.143914e-10 & 5.421869e-02 & 1.820130e-02  & 8.820722e-13 \\
80 & 1.220430e-09 & 4.771142e-02 & 6.050880e-02  & 1.151368e-12 \\
90 & 2.089473e-09 & 4.264961e-02 & 2.066649e-01  & 1.228462e-12 \\
100 & 2.273791e-09 & 3.846343e-02 & 4.420562e-01 & 1.880540e-12 \\
\noalign{\smallskip}
  \hline\hline
\end{tabular}
\end{center}
\end{table}

We also report some information of the interval enclosures and their theoretical lower bounds of $\epsilon$ for some selected $t$  in Table 3.1. Not only $\ubar{\epsilon}$ but also the radii and separation distances of $\mathbb{X}_N$ are also shown in the table.
We can see that the radius of each interval enclosure is far smaller than their separation distance, which means that the assumptions in the above lemmas and theorems are satisfied.
For a fixed $t$, the set of center points for each interval $([\theta]_i,[\phi]_i) \in \mathbb{Z}_N$, $i = 1,\ldots,N$, denoted as $\tilde{X}_N = \{ \tilde{\mx}_{1}\ldots,\tilde{\mx}_{N}\}$, with $\tilde{\mx}_{i}$ defined by (\ref{xcs}), can be regarded as an approximation of a fundamental spherical $t$-design.
For each new point set $\tilde{X}_N$ we compute  $\epsilon$ by (\ref{lowerbound_rectangle}) so that the equality holds exactly. As shown in the table, the values of  $\epsilon$ for each $\tilde{X}_N$ are all very small positive numbers, and grow  with the increase of $t$. This means that $\tilde{X}_N$ which is selected properly from the interval enclosures is a spherical $t_\epsilon$-design.

\begin{remark}
To obtain an approximate spherical $t$-design as accurate as possible, in \emph{\cite{chen2011computational}} the radius of the set of interval enclosures ${\rm rad}(\mathbb{Z}_N)$ is computed to a small scale around $10^{-10}$.
With the introduction of the concept spherical $t_\epsilon$-designs,  it has been shown that any point set selected in $\mathbb{Z}_N$ is a fundamental spherical $t_\epsilon$-design with small enough ${\rm rad}(\mathbb{Z}_N)$ and $\ubar{\epsilon}$ is increasing in  ${\rm rad}(\mathbb{Z}_N)$.
As a result, to reduce the difficulty of computing $\mathbb{Z}_N$, one may relax ${\rm rad}(\mathbb{Z}_N)$ to a larger scale with keeping $\ubar{\epsilon} < 1$ holding. Then all point sets selected in $\mathbb{Z}_N$ are still fundamental spherical $t_\epsilon$-designs with $ 0 \le \epsilon <1$.
\end{remark}

\section{Worst-case errors of quadrature rules using spherical $t_\epsilon$-designs}

In this section we will investigate the worst-case errors for quadrature rules with algebraic accuracy $t$ using spherical $t_\epsilon$-designs in Sobolev spaces which are finite-dimensional rotationally invariant subspaces of $C(\mathbb{S}^2)$.
The bizonal reproducing kernel will be used in the analysis,  which has been wildely applied to analyze approximations on the sphere \cite{brauchart2014qmc,brauchart2007numerical,hesse2005worst,hesse2006cubature,wendland}.
About equal weight quadrature rules, Brauchart et al \cite{brauchart2014qmc} recently develop a way to compute their worst-case errors in Sobolev spaces. In this section, we intend to extend their method to non-equal but still positive weight quadrature rules and show the performance of spherical $t_\epsilon$-designs in numerical integration.

In this section we follow notations and definitions from \cite{brauchart2014qmc}.
Denote the space of square integrable functions on $\mathbb{S}^2$ by $\mathbb{L}_2(\mathbb{S}^2)$. Then it is a Hilbert space with the inner product
\begin{equation}\label{L2_1}
  \langle f,g \rangle _{\mathbb{L}_2(\mathbb{S}^d)} = \int_{\mathbb{S}^2}f(\mx)g(\mx){\rm d}\omega(\mx), \quad f,g \in \mathbb{L}_2(\mathbb{S}^2),
\end{equation}
and the induced norm as
\begin{equation}\label{L2_2}
  \|f\|_{\mathbb{L}_2(\mathbb{S}^2)} = \left( \int_{\mathbb{S}^2}|f(\mx)|^2{\rm d }\omega(\mx) \right)^{\frac12}, \quad f \in \mathbb{L}_2(\mathbb{S}^2).
\end{equation}
The Sobolev space $\mathbb{H}^s(\mathbb{S}^2)$ can be defined for $s\geq 0$ as the set of all functions $f \in \mathbb{L}_2(\mathbb{S}^2)$ with whose Laplace-Fourier coefficients
\begin{equation}\label{sobolev1}
  \hat{f}_{\ell,k} = \left\langle f , Y_{\ell,k} \right\rangle_{\mathbb{L}_2(\mathbb{S}^2)} = \int_{\mathbb{S}^2}f(\mx)Y_{\ell,k}(\mx){\rm d}\omega(\mx),
\end{equation}
satisfying
 \begin{equation}\label{sobolev2}
   \sum_{\ell = 0}^{\infty} \sum_{k = 1}^{2\ell +1}(1+\lambda_\ell)^s \left| \hat{f}_{\ell,k} \right|^2 < \infty,
 \end{equation}
where $\lambda_\ell = \ell (\ell + 1 )$. Obviously, by letting $s = 0$ we can obtain $\mathbb{H}^0(\mathbb{S}^2) = \mathbb{L}_2(\mathbb{S}^2)$.
Then the norm of $\Hsst$ can be defined as
\begin{equation}\label{norm_s}
  \|f\|_{\mathbb{H}^s} =\left[ \sum_{\ell = 0}^{\infty}\sum_{k =1}^{2\ell+1}\frac{1}{\alpha_\ell^{(s)}} \hat f_{\ell,k} ^2\right]^{\frac12},
\end{equation}
where the sequence of positive parameters $\alpha_\ell^{(s)}$ satisfies
\begin{equation}\label{alpha_ls_bound}
  \alpha_\ell^{(s)} \sim (1+\lambda_\ell)^{-s} \sim (\ell+1)^{-2s}.
\end{equation}
Correspondingly, the inner product of $\mathbb{H}^s(\mathbb{S}^2)$ can be defined as
\begin{equation}\label{inner_prod}
\langle f,g \rangle _{\mathbb{H}^s}  = \sum_{\ell = 0}^{\infty}\sum_{k = 1}^{2\ell+1} \frac{1}{\alpha_\ell^{(s)}} \hat f_{\ell,k} \hat g_{\ell,k}.
\end{equation}
For a point set $X_N$ and a weight vector $\mw$, we define the numerical quadrature rule and the integral of a function $f$ on $\St$ as
\begin{equation}\label{cubature_integral_epsi}
  Q[X_N,\mw](f):= \sum_{j =1}^{N}\frac {w_j} {4 \pi}f(\mx_j), \quad I(f) := \int_{\mathbb{S}^2}f(\mx){\rm d}\omega(\mx),
\end{equation}
The worst-case error of the quadrature rule $Q[X_N,\mw]$ on $\Hsst$ can be defined as \cite{brauchart2014qmc,hesse2005worst}
\begin{equation}\label{cubature_integral_epsi1}
E_{s}(Q[X_N,\mw]) := \sup\big\{ \left|Q[X_N,\mw](f) - I(f) \right|:f\in \Hsst, \|f\|_{\mathbb{H}^s} \leq 1 \big\}.
\end{equation}
The Riesz representation theorem and the additional theorem assure the existence of a reproducing kernel of
\begin{eqnarray}\label{Ks}
 \nonumber K_s(\mx,\my) &=& \sum_{\ell = 0}^{\infty}(2\ell+1)\alpha_\ell^{(s)}P_\ell(\mx \cdot \my) \\
   &=& \sum_{\ell = 0}^{\infty}\sum_{k = 1}^{2\ell+1} \alpha_\ell^{(s)} Y_{\ell,k}(\mx) Y_{\ell,k}(\my).
\end{eqnarray}
Together with the property of reproducing kernel $K_s(\cdot,\cdot)$ defined in (\ref{Ks}) and the addition theorem,  it is shown in \cite{hesse2005worst} that
\begin{eqnarray*}
  \big(E_s(Q[X_N,\mw])\big)^2  &=& \left[ \sup_{\mbox{\tiny$\begin{array}{c}
f \in \Hsst\\
\|f\|_{\mathbb{H}^s} \leq 1\\
\end{array}$}} \left|Q[X_N,\mw](f) -I(f)\right|\right]^2 \\
   &=&  \left\| \sum_{i = 1}^{N} \frac {w_i} {4 \pi} K_s(\cdot,\mx)  -  \Is K_s(\cdot,\mx) {\rm d}\omega(\mx) \right\|_{\mathbb{H}^s}^2.
\end{eqnarray*}
Assume that $Q[X_N,\mw]$ has algebraic accuracy $t$. Then with the following equality
\begin{equation*}
  \Is K_s(\mx,\cdot)d\omega(\mx) = \alpha_0^{(s)},
\end{equation*}
the worst-case error could be reformulated as
\begin{eqnarray}
 \nonumber (E_s(Q[X_N,\mw]))^2  &=&  \left[\sum_{\ell=1}^{\infty}\sum_{k=1}^{2\ell+1}\alpha_\ell^{(s)}\left(\sum_{i=1}^{N}\frac {w_i} {4 \pi}Y_{\ell,k}(\mx_i)\right)^2\right]\\
 &=& \sum_{i=1}^N \sum_{j=1}^N \frac{w_iw_j}{16\pi^2} \sumlo \sumk \aals Y_{\ell,k}(\mx_i)Y_{\ell,k}(\mx_j).
\end{eqnarray}
Reproducing kernels for $\Hsst$ for $s > 1$ can be constructed utilizing powers of distances, provided the power $2s - 2$ is not an even integer. Indeed, it is known (cf., e.g., \cite{baxter2001radial,brauchart2007numerical}) that the signed power of the distance, with sign $(-1)^{L+1}$ with $L:=L(s):= \lfloor s - 1 \rfloor $,
has the following Laplace-Fourier expansion
\begin{equation}\label{lap_fou_expansion}
  (-1)^{L+1} |\mx -\my|^{2s-2} = (-1)^{L+1} V_{2-2s}(\St) + \sum_{\ell =1}^{\infty} a_{\ell}^{(s)} (2\ell +1)P_\ell (\mx \cdot \my),
\end{equation}
where
\begin{equation}\label{V_d2s}
  V_{2-2s} (\St) := \int_{\St} \int_{\St} |\mx -\my |^{2s-2} {\rm d}\omega(\mx){\rm d}\omega{(\my)} = 2^{2s-1} \dfrac{\Gamma(3/2)\Gamma(s)}{\sqrt{\pi}\Gamma(1+s)},
\end{equation}
\begin{equation}\label{alpha_ls}
  a_\ell^{(s)} := V_{2-2s}(\St) \dfrac{(-1)^{L+1}(1-s)_\ell}{(1+s)_\ell}, \ \ell \geq 1,
\end{equation}
and
$$
\frac{(1-s)_\ell}{(1+s)_\ell} := \frac {\Gamma(1+s)}{\Gamma(1-s)}\frac {\Gamma(\ell+1-s)}{\Gamma(\ell+1+s)} \sim \frac {\Gamma(1-s)}{\Gamma(1+s)} \ell^{-2s} \sim \ell^{-2s}.
 $$
Thus we have
\begin{eqnarray}\label{dis_reproducingkernel}
  (-1)^{L+2}(V_{2-2s}(\St)-|\mx-\my|^{2s-2}) &=& \sumlo \als (2\ell+1)P_\ell(\mx\cdot\my) \\
   &=&  \sumlo \als \sumk Y_{\ell,k}(\mx)Y_{\ell,k}(\my).
\end{eqnarray}
Note that for $\als$ we have
\begin{equation}\label{als_eq}
  \als \sim 2^{2s-1} \dfrac{\Gamma(\frac{3}2)\Gamma(s)}{\sqrt{\pi}(-1)^{L+1}\Gamma(1 +s)}\ell^{-2s} \quad {\rm as} \  \ell \to \infty,
\end{equation}
and when $1< s \leq 2$, which means $L = L(s) =0$, we have $\als > 0 $ for all $\ell = 1,\ldots$. Therefore, we regard the left hand side of (\ref{dis_reproducingkernel}) as the reproducing kernel of $\Hsst$, which is
$$
  K_s(\mx,\my) = V_{2-2s}(\St)-|\mx-\my|^{2s-2},$$
   and then we obtain
\begin{equation}\label{E_s_small}
  (E_s(Q[X_N,\mw]))^2 = \sum_{i=1}^N \sum_{j=1}^N \frac{w_iw_j}{16\pi^2} (V_{2-2s}(\St)-|\mx_i-\mx_j|^{2s-2}).
\end{equation}
For the case $s > 2$, we know that $\als >0$ does not hold for all $\ell = 1,\ldots$. In this situation,  we let
 $$
 K_s(\mx,\my) = (1-(-1)^{L+1})V_{2-2s}(\St) + \mathcal{Q}_{L}(\mx \cdot \my) + (-1)^{L+1} |\mx - \my|^{2s-2},
   $$
 with
 $$
\mathcal{Q}_{L}(\mx \cdot \my) : = \sum_{\ell=1}^{L} ((-1)^{L+1-\ell}-1)a_\ell^{(s)}(2\ell+1)P_\ell(\mx\cdot \my), \quad \mx,\my \in \St,
$$
which changes the signs if the negative coefficients $\als$ in (\ref{lap_fou_expansion}). Hence the worst-case error on $\Hsst$ with $s>2$ can be represented as
\begin{equation}\label{E_s_large}
\begin{aligned}
( E_s(Q[X_N,\mw]))^2 =  \sum_{i=1}^N \sum_{j=1}^N \frac{w_iw_j}{16\pi^2} &\left( \mathcal{Q}_{L}(\mx_i \cdot \mx_j) + (-1)^{L+1} |\mx_i - \mx_j|^{2s-2} \right.\\
& \left. -(-1)^{L+1}V_{2-2s}(\St) \right).\\
\end{aligned}
\end{equation}

In what follows we will compute the worst-case errors of quadrature rules using spherical $t_\epsilon$-designs with algebraic accuracy $t$. In this experiment we choose $\epsilon =0.1$ for spherical $t_\epsilon$-designs and use (\ref{Yw}) to find a spherical $t_\epsilon$-design, which is a system of nonlinear equations. The system can be solved by minimizing its least squares form using a smoothing trust-region filter method proposed in \cite{annual14}.
Note that the number of points needed for constructing spherical $t_\epsilon$-designs may decrease with the increase of $\epsilon$.
Thus in the computation of spherical $t_\epsilon$-designs we always attempt to find the one with  a possible minimal number of points, denoted as $N(t,\epsilon)$.
The detailed process for finding spherical $t_\epsilon$-designs can be found in \cite{annual14}.
In the numerical test of computation of spherical $t_\epsilon$-designs it is found that a possible minimal number of points
satisfies $\lceil(t+1)^2/3\rceil +1\le N(t,\epsilon) \le \lceil(t+2)^2/2\rceil+1$.
In the numerical test in current and next sections, the spherical $t_\epsilon$-designs are chosen with $N = N(t,\epsilon)$.

\begin{figure}[!h]\label{wce_fig}
	\centering
	\subfigure[$s = 1.5$]{ \label{wce_small_fig} 
		\includegraphics[width=2.4in]{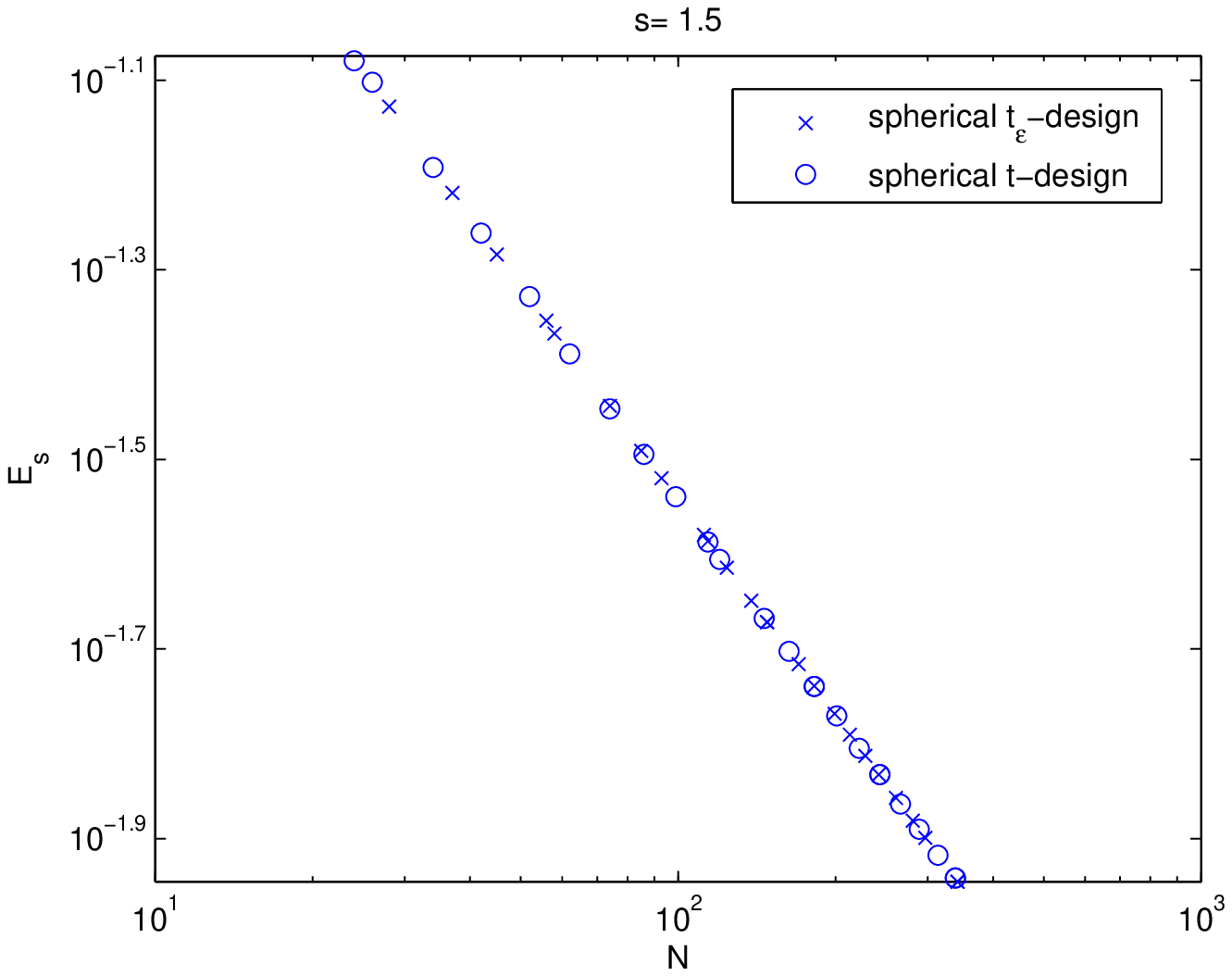}}
	\subfigure[$s = 5.5$]{ \label{wce_large_fig} 
		\includegraphics[width=2.4in]{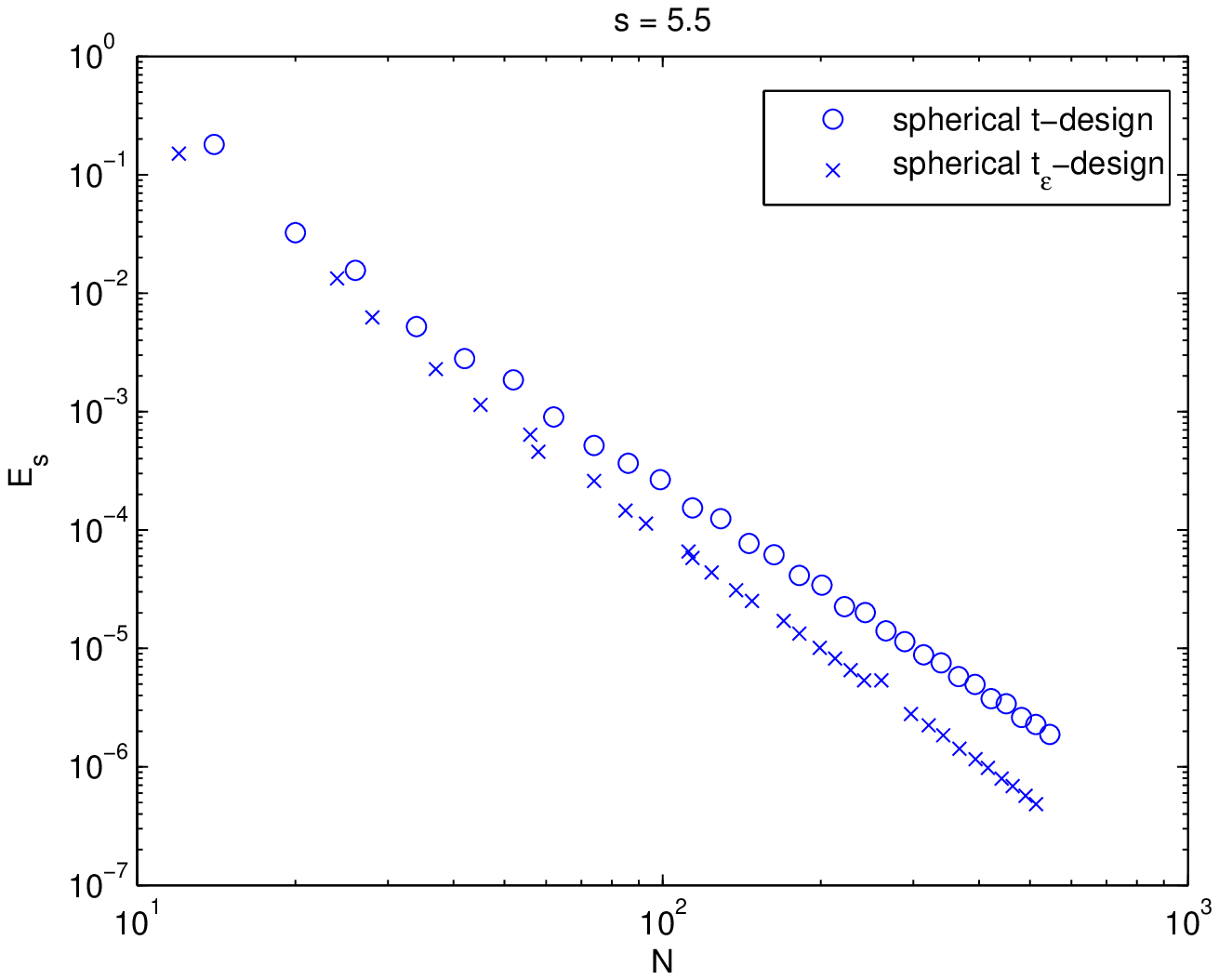}}
	\caption{Worst-case errors for spherical $t_{0.1}$-designs and spherical $t$-designs}
\end{figure}

The worst-case errors of quadrature rules using spherical $t_{0.1}$-designs in $\Hsst$ for $s = 1.5 $  are illustrated in Fig. \ref{wce_small_fig}.
For comparison, the worst-case errors for quadrature rules using approximate spherical $t$-designs computed in \cite{sloan2009variational}  will also be implemented.
For all spherical $t_{0.1}$-designs, the worst-case error with $s = 1.5$ is calculated using (\ref{E_s_small}) and the distance kernel, and for spherical $t$-designs the worst-case errors are calculated by relative results in \cite{brauchart2014qmc}.
From the figure we can see that in this case, the computed worst-case errors of approximate spherical $t$-designs and spherical $t_{0.1}$-designs essentially lie on the same curve, which remains as a conjecture that the worst-case errors of both spherical $t$-designs and spherical $t_\epsilon$-designs decay in the same speed with respect to the number of points in the case $s<2$ on $\St$.
Figure \ref{wce_large_fig} plots the worst-case errors for  both spherical $t$-designs and spherical $t_{0.1}$-designs with $s = 5.5$.
From the figure we can see that the worst-case errors of spherical $t_{0.1}$-designs decay faster than the ones of spherical $t$-design with respect to the number of points.

\section{Polynomial approximation on the sphere using spherical $t_\epsilon$-designs}

\subsection{Regularized weighted least squares approximation using spherical $t_\epsilon$-designs}

In this section we consider the restoration of a continuous function  $f \in C(\St)$ from its noisy values $f^\delta$ given at $N$ points $X_N = \{ \mx_1, \ldots, \mx_N\} \subset \mathbb{S}^2$ by the $l_2-l_1$ regularized weighted discrete least squares form
\begin{equation}\label{weighted}
 \min_{\alpha_{\ell,k} \in \bR} \frac12 \sum_{j = 1}^N \mu_j(\sum_{\ell = 0 }^{L}\sum_{k =1}^{2k+1}\alpha_{\ell,k}Y_{\ell,k}(\mx_j) - f^\delta(\mx_j))^2 + \lambda  \sum_{\ell = 0}^L\sum_{k =1}^{2\ell+1}|\beta_{\ell,k} \alpha_{\ell,k}|
\end{equation}
where $\mu_j > 0$, $j = 1, \ldots, N$ are the weights for each term of the least squares model, $\lambda > 0 $ is the regularization parameter, and $\beta_{\ell,k} \ge 0 $, $\ell = 0,\ldots,L, \  k  = 1,\ldots,2\ell +1 $ are usually chosen with the meaning of certain polynomial operators such as Laplace-Beltrami operator and filtered operator \cite{an2012regularized,sloan2012filtered}. In \cite{pereverzyev2014parameter} both a priori choice based physical reason in satellite gravity gradiometry problem and a posteriori choice based on reproducing kernel theory are considered to choose $\beta_{\ell,k}$.

Note that  $\{ Y_{\ell,k}, \  k = 1,\ldots,2\ell+1, \ \ell =0,\ldots,L\}$ is a basis of $\mathbb{P}_L$. Problem (\ref{weighted}) is to find a good approximation of $f$ in $\mathbb{P}_L$ in the form
$$p_{L,N}(\mx)= \sum_{\ell = 0}^{L}\sum_{k = 1}^{2\ell+1} \alpha_{\ell, k} Y_{\ell,k}(\mx).$$
Let the entries of  matrix $\mathbf{Y}_L \in \mathbb{R}^{N \times (L+1)^2 }$ be
$$(\mathbf{Y}_L)_{i,\ell^2+k} = Y_{\ell,k}(\mx_i), \ \  i = 1,\ldots, N, \ \ell= 0,\ldots,L, \  k = 1,\ldots,2\ell+1, $$
and
$\mathbf{f}^\delta = (f^\delta(\mx_1),\ldots,f^\delta(\mx_N))^T$.
Problem (\ref{weighted}) can be reformulated  as
\begin{equation}\label{weighted_2}
 \quad \min_{\alpha \in \mathbb{R}^{(L+1)^2}} \frac12\|\mathbf{\Lambda}^{\frac12}(\mathbf{Y}_L\alpha - \mathbf{f}^\delta)\|^2_2 + \lambda  \| \mathbf{D}\alpha\|_1,
\end{equation}
where
$$
\mathbf{\Lambda} = \left[
                            \begin{array}{ccc}
                              \mu_1 &  &  \\
                               & \ddots &    \\
                               &  &   \mu_N \\
                            \end{array}
                          \right] \in \bR^{N\times N},
$$
and $\mathbf{D}$ is a diagonal matrix satisfying $\mathbf{D}_{\ell^2+k,\ell^2+k} = \beta_{\ell,k}$ with $\beta_{\ell,k} \ge 0$. For polynomial approximation on the sphere, an $l_2$-regularized weighted least squares model has also been considered \cite{an2012regularized,pereverzyev2014parameter}
\begin{equation}\label{weighted_3}
\min_{\alpha \in \mathbb{R}^{(L+1)^2}} \frac12 \|\mathbf{\Lambda}^{\frac12}(\mathbf{Y}_L\alpha - \mathbf{f}^\delta)\|^2_2 + \lambda  \| \mathbf{D}\alpha\|_2^2.
\end{equation}
The regularization of this model is of $l_2$ norm, which can be seen as a measure of energy.
It is known that the $l_1$ regularization has desirable properties in approximation of nonsmooth continuous functions.
An $l_1$ regularization term is preferable to be considered here.
By choosing a suitable penalization term, the $l_2-l_1$ regularized model  is usually supposed to  achieve a more sparse solution than the $l_2-l_2$ regularized one, which means that the target function is approximated by less basis spherical polynomials.
Additionally, for functions which are globally continuous but locally non-differentiable on the sphere, the $l_2-l_1$ regularization is  better than the $l_2-l_2$ regularization.

\begin{theorem}
 Let $X_N$ be a spherical $t_\epsilon$-design and $\mw$ be the vector of weights satisfying  (\ref{a_t_design2}) and  (\ref{a_t_design}) with respect to $X_N$.
Let $L \geq 0$.
 For model (\ref{weighted_2}) set $\mu_j = w_j$ for $j =  1,\ldots,N$. Then
 \begin{equation}\label{H_L}
   \mathbf{H}_L = \mathbf{Y}^T_L \mathbf{\Lambda}\mathbf{Y}_L = \mathbf{I}_{(L+1)^2},
 \end{equation}
and (\ref{weighted_2}) has the unique solution
 \begin{equation}\label{alpha_lk}
   \alpha_{\ell,k} = \max\{0,s_{\ell,k}-\lambda \beta_{\ell,k}\}+\min\{0,s_{\ell,k}+\lambda\beta_{\ell,k}\},
 \end{equation}
 for $\ell = 0, \ldots,L$, $k =1,\ldots,2k+1$, where $s_{\ell,k} =\sum_{i =1}^{N}w_iY_{\ell,k}(\mx_i)f^\delta(\mx_i)$.
\end{theorem}

\begin{proof} Note that when $X_N$ is a spherical $t_\epsilon$-design,
 \begin{eqnarray}\label{HL_identity}
                    \nonumber        \left(\mathbf{H}_L\right)_{\ell^2+ k , (\ell^\prime)^2+ k^\prime} &=& \sum_{i = 1}^{N}w_i Y_{\ell,k}(\mx_i) Y_{\ell^\prime,k^\prime}(\mx_i) \\
                             &=& \int_{\mathbb{S}^2} Y_{\ell,k}(\mx) Y_{\ell^\prime,k^\prime}(\mx) {\rm d} \omega(\mx)
=\delta_{\ell \ell^\prime} \delta_{k k^\prime},
                          \end{eqnarray}
where the third equality is established by the orthonormality of spherical harmonics.
Problem (\ref{weighted_2}) is strictly convex by the fact that $\mathbf{H}_L$ is nonsingular and so it has a unique optimal solution.
Since $A(\alpha) = \frac12 \|\mathbf{\Lambda}^{\frac12}(\mathbf{Y}_L\alpha - \mathbf{f}^\delta)\|^2_2$ is strictly differentiable,
by deriving the first optimality condition of (\ref{weighted_2}) and Corollary 1 in \cite[Section 2.3]{clarke1990optimization}, we obtain that its unique optimal solution satisfies
 \begin{equation}\label{1st_optimality}
   0 \in \mathbf{H}_L\alpha - \mY^T_L\mathbf{W}\mathbf{f}^\delta + \lambda \partial(\|\mathbf{D}\alpha\|_1),
 \end{equation}
where $\partial(\cdot)$ denotes the subdifferential. By (\ref{HL_identity}) which implies $\mathbf{H}_L = \mathbf{I}_{(L+1)^2}$ and the fact that $\mathbf{D}$ is diagonal, problem (\ref{1st_optimality})  is separable and thus $\alpha$ is a solution of (\ref{1st_optimality}) if and only if it is a solution of
\begin{equation}\label{seperate_problem}
0 \in \alpha_{\ell,k} - s_{\ell,k} + \lambda  \beta_{\ell,k} \partial | \alpha_{\ell,k}|, \quad \ell = 0,\ldots,L, k = 1,\ldots,2\ell+1.
\end{equation}
Denote $\tau_{\ell,k} = \partial|\alpha_{\ell,k}|$ and hence $ - 1 \le \tau_{\ell,k} \le 1$.
Let $\alpha_{\ell,k}^\ast$ be the optimal solution of (\ref{seperate_problem}) with corresponding $\ell$ and $k$ and hence
\begin{equation}\label{alpha_lk_ast}
\alpha_{\ell,k}^\ast  = s_{\ell,k} - \lambda \beta_{\ell,k} \tau_{\ell,k} \quad {\rm with}  \ \tau_{\ell,k} \in [-1,1].
\end{equation}
When $s_{\ell,k} > \lambda\beta_{\ell,k}$ we can set $\tau_{\ell,k} = 1$ and obtain
$$ \alpha^\ast_{\ell,k} = s_{\ell,k} - \lambda \beta_{\ell,k} > 0,$$
which together with $\beta_{\ell,k} \ge 0$ satisfies  (\ref{alpha_lk}) and (\ref{alpha_lk_ast}). When $s_{\ell,k} < -\lambda\beta_{\ell,k}$ similarly we set $\tau_{\ell,k} = -1$ and get
$$ \alpha^\ast_{\ell,k} = s_{\ell,k} + \lambda \beta_{\ell,k} < 0,$$
which also satisfies  (\ref{alpha_lk}) and (\ref{alpha_lk_ast}).
Then when $s_{\ell,k} \in [- \lambda\beta_{\ell,k},\lambda\beta_{\ell,k}]$ we set $\tau_{\ell,k} = \frac{s_{\ell,k}}{\lambda \beta_{\ell,k}} \in [-1,1]$ and get that
$$ \alpha^\ast_{\ell,k} = 0,$$
which also satisfies  (\ref{alpha_lk}) and (\ref{alpha_lk_ast}).
Hence the theorem is proved.
\end{proof}

Denote the approximation residual as  $A(\alpha) = \sum_{j = 1}^{N} (p_{L,N}(\mx_j) - f^\delta(\mx_j))^2$. Let $\alpha^\ast(\lambda)$ be the optimal solution of (\ref{weighted_2}) with different regularized parameters $\lambda$. The following proposition indicates that $A(\alpha^\ast(\lambda))$ is monotonically  increasing with respect to $\lambda$.

\begin{proposition}
Let $X_N^\epsilon$ be a spherical $t_\epsilon$-design with $t \ge 2 L$ and $\mu_j = w_j$ for $j =  1,\ldots,N$.
Then $A(\alpha^\ast(\lambda))$ is  increasing in $\lambda$.
\end{proposition}

\begin{proof}
Let $\lambda$, $\tilde{\lambda}$ be given with $0 < \lambda \le \tilde{\lambda}$ and denote the optimal solution of problem (\ref{weighted_2}) with $\lambda$, $\tilde{\lambda}$ as $\alpha^\ast$, $\tilde{\alpha}^\ast$ respectively.
Denote $E(\lambda,\alpha) = \lambda\|\mathbf{D}\alpha\|_1$ and the minimization property of (\ref{weighted_2}) for $\lambda$ gives
\begin{equation}\label{proposition4.1.1}
  A(\alpha^\ast) + E(\lambda,\alpha^\ast) \le A(\tilde\alpha^\ast) + E(\lambda,\tilde\alpha^\ast),
\end{equation}
which implies that
\begin{equation}\label{proposition4.1.2}
  A(\alpha^\ast) - A(\tilde\alpha^\ast) \le   E(\lambda,\tilde\alpha^\ast) -  E(\lambda,\alpha^\ast).
\end{equation}
From (\ref{alpha_lk}) we have
\begin{equation}\label{alpha_lk2}
  \alpha^\ast_{\ell,k} =
  \left\{
                      \begin{array}{cl}
                        s_{\ell,k} - \lambda\beta_{\ell,k}, & \quad\quad s_{\ell,k} > \lambda\beta_{\ell,k}\\
                        s_{\ell,k} + \lambda\beta_{\ell,k}, & \quad\quad s_{\ell,k} < -\lambda\beta_{\ell,k}\\
                        0. &  \quad\quad- \lambda\beta_{\ell,k} \le s_{\ell,k} \le \lambda\beta_{\ell,k} \\
                      \end{array}
                    \right.
\end{equation}
Since $\lambda\beta_{\ell,k} \ge 0$, we have $|\alpha^\ast_{\ell,k}| = \max(0,|s_{\ell,k}| - \lambda \beta_{\ell,k})$.
Together by the fact that
$$
E(\lambda,\alpha^\ast) =  \lambda \sum_{\ell = 0} ^{L} \sum_{k =1}^{2\ell +1}  \beta_{\ell,k} |\alpha^\ast_{\ell,k}|,
$$
we have
$$
|\tilde\alpha^\ast_{\ell,k}| = \max(0,|s_{\ell,k}|- \tilde\lambda \beta_{\ell,k}) \le \max(0,|s_{\ell,k}|- \lambda \beta_{\ell,k}) = |\alpha^\ast_{\ell,k}|.
$$
Hence it is obtained that $E(\lambda,\tilde\alpha^\ast) \le E(\lambda,\alpha^\ast)$. Together with (\ref{proposition4.1.1}) we complete the proof.
\end{proof}

\subsection{Numerical experiments}

In this subsection we report the numerical results to test the efficiency of the $l_2-l_1$ regularized model (\ref{weighted_2}) using spherical $t_\epsilon$-designs.

\textbf{Example 5.1.} In the first numerical test, the target function is selected as spherical polynomials with degree no higher than $L$. Obviously using  both models (\ref{weighted_2}) and (\ref{weighted_3})  the target function can be exactly restored when $\lambda = 0 $ and the data $\mathbf{f}^\delta$ is noise free and  the optimal values of the two models equal to  0 in such case. However, due to the noise in the data vector $\mathbf{f}^\delta$, it is necessary to use the  regularization models.

In this experiment we will use the spherical $t_{0.1}$-designs which is calculated by solving a system of nonlinear equation (\ref{Yw}) and the approximate spherical $t$-designs proposed in \cite{sloan2009variational} as the point set for polynomial approximation. Both the uniform errors and $\mathbb{L}_2$ errors are recorded to measure the approximation quality.
We choose a large-scaled and well distributed point set $X_t \subset \St$ to be the test set and use it to estimate the errors. Then the uniform error and $\mathbb{L}_2$ error of the approximation  are estimated by
\begin{equation}\label{uniform_error}
  \|f - p_{L,N}\|_{C(\St)} \approx \max_{\mx_i \in X_t}|f(\mx_i)-p_{L,N}(\mx_i)|,
\end{equation}
and
\begin{equation}\label{L2_error}
  \|f - p_{L,N}\|_{\mathbb{L}_2} \approx \left(\dfrac{4\pi}{N_t} \sum_{i =1}^{N_t}(f(\mx_i)-p_{L,N}(\mx_i))^2\right)^{\frac12},
\end{equation}
where $N_t$ denotes the number of points $\mx_i$ in $X_t$. In this experiment, we choose $X_t$ to be an equal area partitioning point set \cite{saff} with $10^5$ points.
The  matrix $\mathbf{D}$ in the experiment is always selected as $\beta_{\ell,k} = \ell(\ell+1)$ for $\ell = 0,\ldots,L, \ k = 2\ell+1$, inspired  by the Laplace-Beltrami operator, see \cite{an2012regularized}.

\begin{figure}[!h]\label{2models_polynomial}
 \centering
\subfigure[Uniform Errors with different $\lambda$ for $\delta = 0.1$]{ \label{rand1_2models_polynomial_l2} 
\includegraphics[width=2.4in]{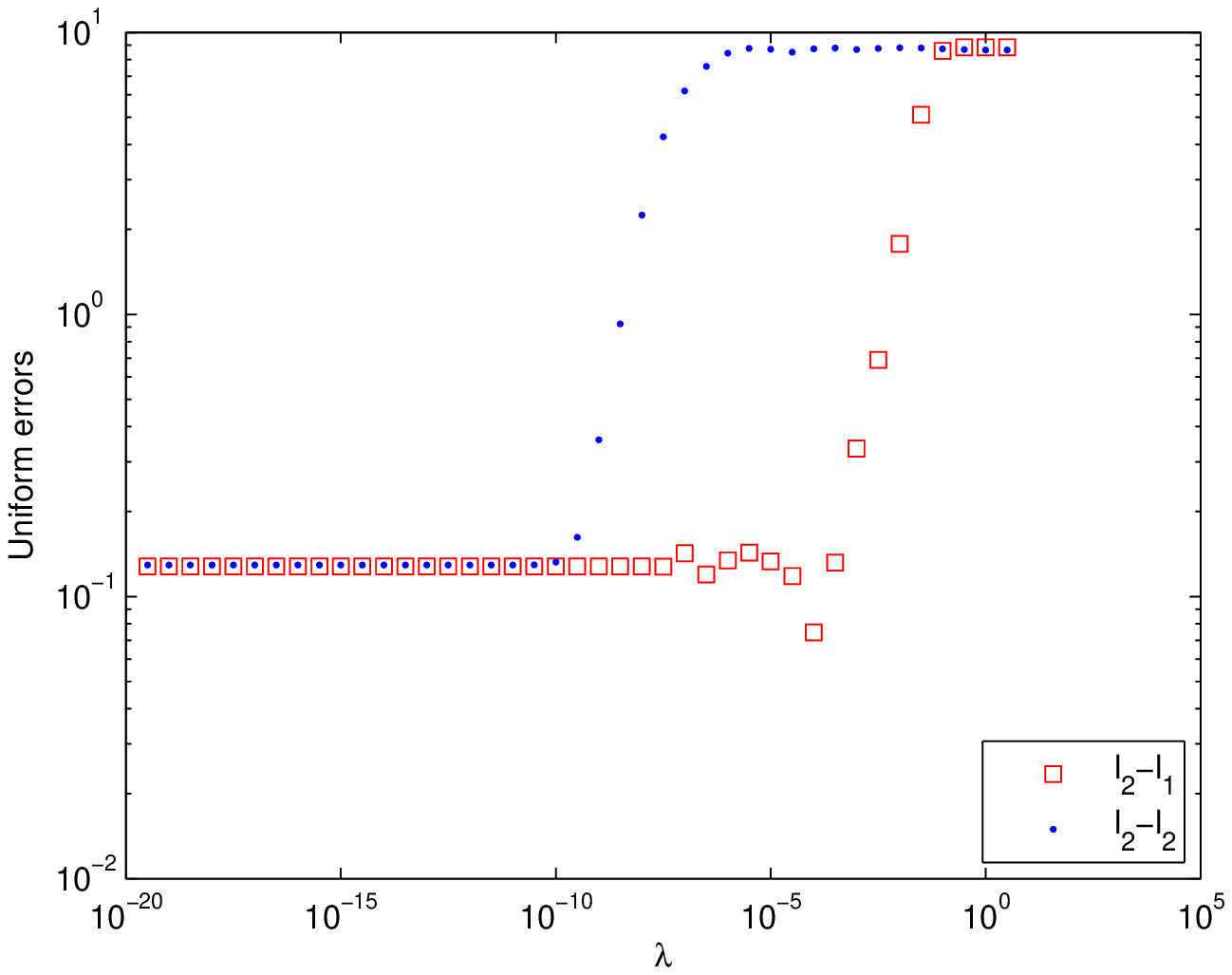}}
\subfigure[$\mathbb L_2$ errors with different $\lambda$ for $\delta = 0.1$]{ \label{rand1_2models_polynomial_uniform} 
\includegraphics[width=2.4in]{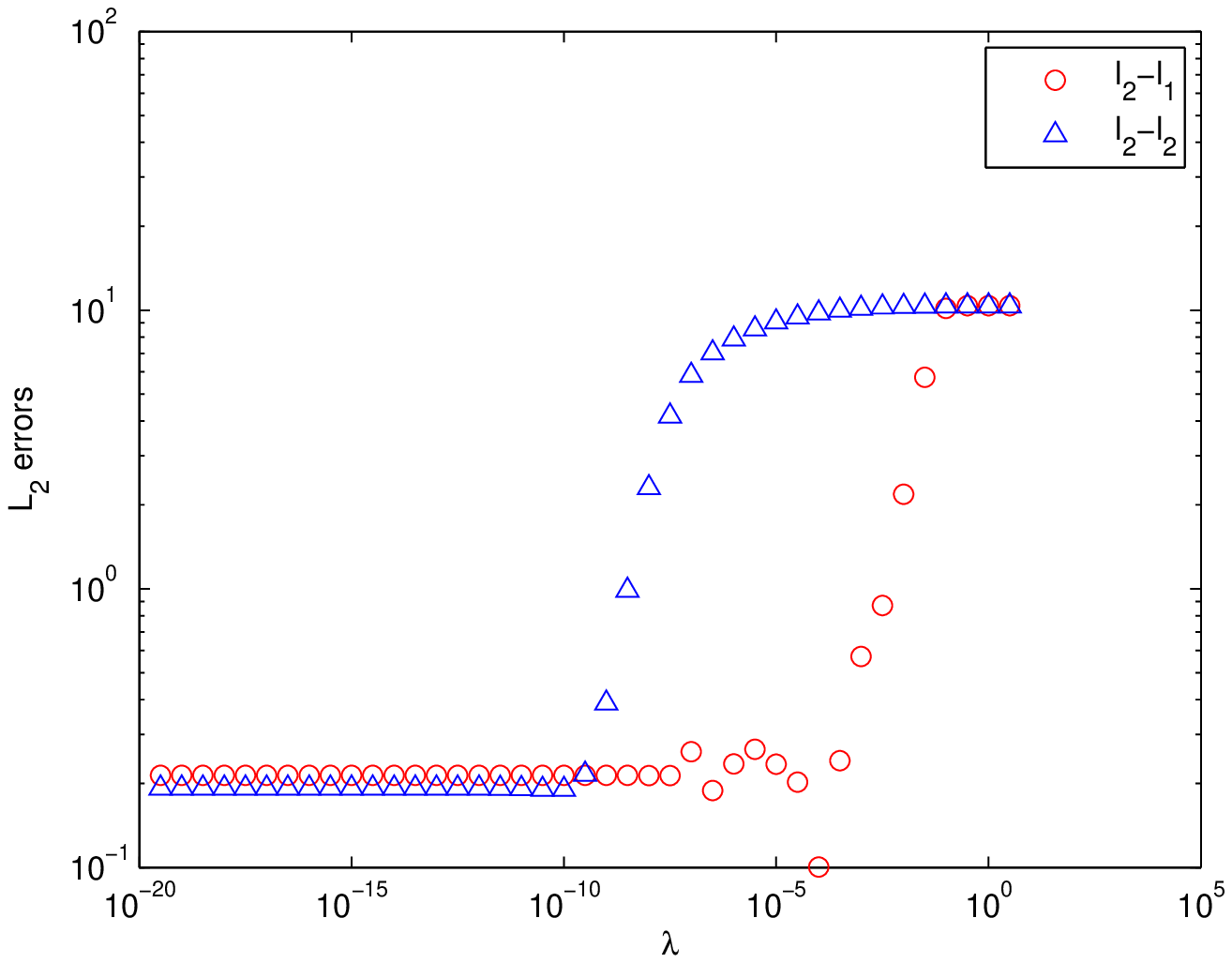}}
\subfigure[Minimal uniform errors with different noise scales $\delta$]{ \label{2models_polynomial_multirand_uniform} 
\includegraphics[width=2.4in]{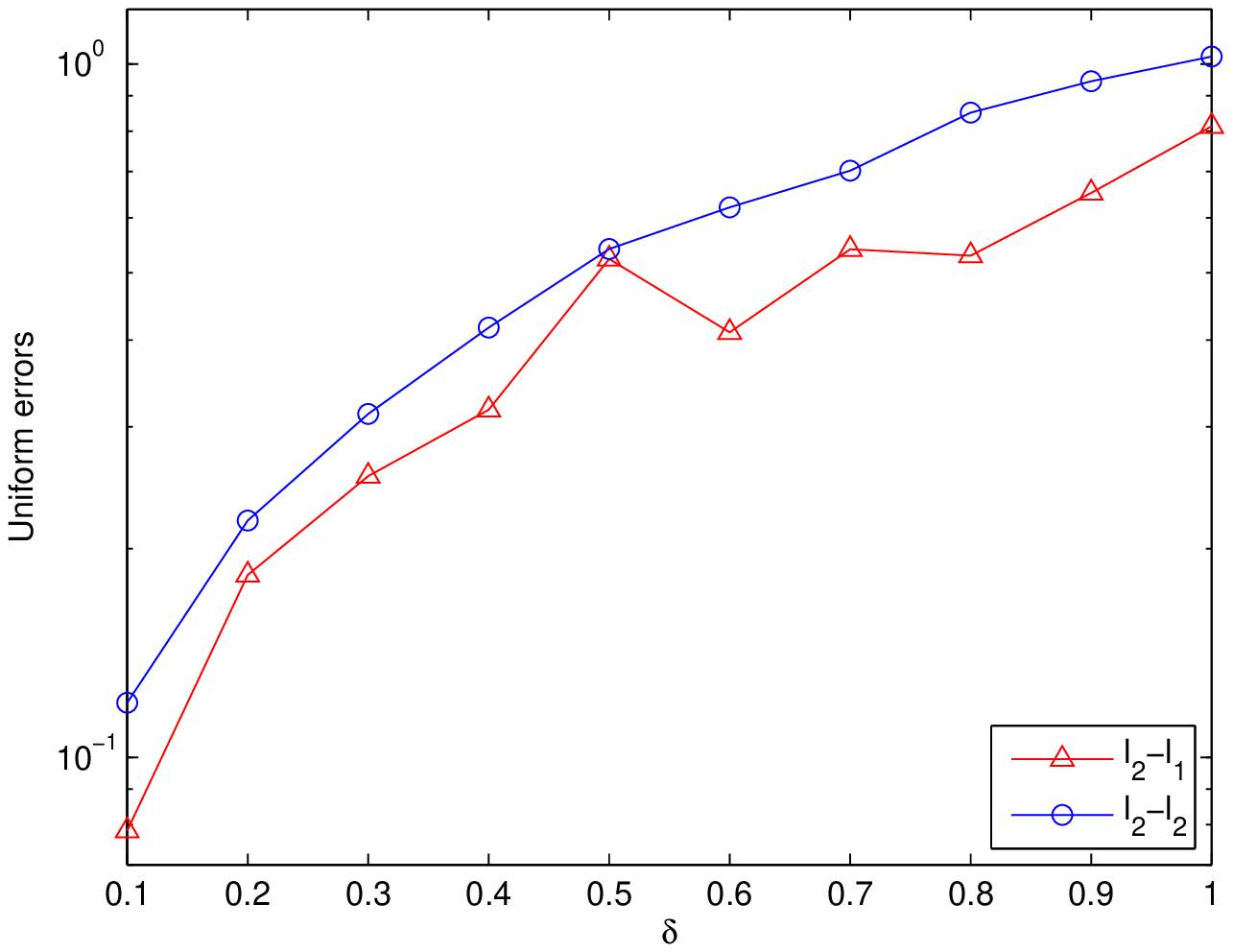}}
\subfigure[Minimal $\mathbb L_2$ errors with different noise scales $\delta$]{ \label{2models_polynomial_multirand_l2} 
\includegraphics[width=2.4in]{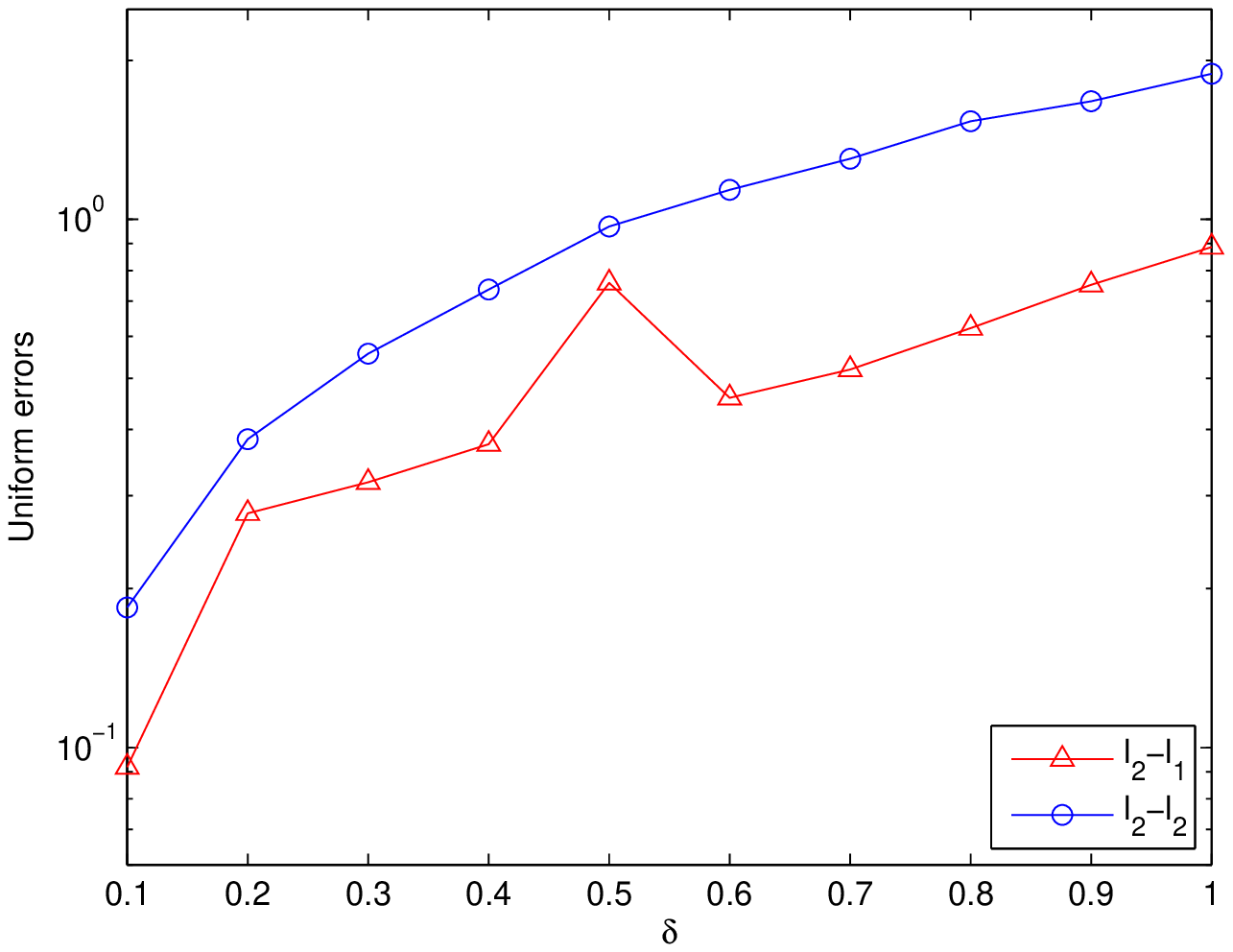}}
\caption{Errors for restoring $18$-degree polynomial}
\end{figure}

Fig. 5.1 shows the approximation errors using both $l_2-l_1$ model (\ref{weighted_2}) and $l_2-l_2$ model (\ref{weighted_3}) with different $\lambda$ and different noise scales $\delta$.
The noise of the data $\mathbf{f}^\delta$ obeys a uniform distribution in $[-\delta,\delta]$.
In this numerical experiment a spherical $37_{0.1}$-design with only 514 points, which is much less than $\lceil \frac {(t+1)^2}2\rceil$, is applied to approximate a randomly generated spherical polynomial with degree $ \lfloor\frac{37}2\rfloor = 18$ (The polynomial is generated with all its Fourier coefficients obeying the standard normal distribution).
The regularization parameter $\lambda$ is chosen from $10^{-20}$ to $10^{0.5}$. Fig. 5.1 (a)(b) give the errors of the approximation with different $\lambda$ for $\delta = 0.1$ using the two models.
From the two sub-figures it can be seen that model (\ref{weighted_2}) can restore the $18$-degree polynomial more accurately than model (\ref{weighted_3}).  The  minimal error with respect to different $\lambda$ can be achieved  at about $\lambda =  10^{-6}$.
Fig. 5.1 (c)(d) show the errors of the restoration results with different noise scales. It can be seen that the model (\ref{weighted_2}) performs better in each noise scale than (\ref{weighted_3}).\\

\textbf{Example 5.2.} In the second numerical experiment we test the numerical performance of model (\ref{weighted_2}) using spherical $t_\epsilon$-designs and spherical $t$-designs. We select the Franke function \cite{renka1988multivariate}
\begin{equation}\label{frankie_function}
\begin{split}
f_1(\mathbf{x})= & f(x,y,z) = 0.75 \exp(-(9x-2)^{2}/4-(9y-2)^{2}/4-(9z-2)^{2}/4) \\
& +0.75 \exp(-(9x+1)^{2}/49-(9y+1)/10-(9z+1)/10) \\
& +0.5\exp(-(9x-7)^{2}/4-(9y-3)^{2}/4-(9z-5)^{2}/4) \\
& -0.2\exp(-(9x-4)^{2}-(9y-7)^{2}-(9z-5)^{2}), \  (x,y,z)\in \mathbb{S}^{2}
\end{split}
\end{equation}
to be the target function which is not a spherical polynomial but continuously differentiable on the whole sphere.
We set $\epsilon = 0.1$ and also $\delta = 0.1$ in this experiment and the scheme of choosing $\lambda$ is the same as in Example 5.1.
For spherical $t_{0.1}$-designs we select those point sets constructed with possible least  points.
As is mentioned above, a spherical $t_{0.1}$-design may be constructed using less than $\lceil \frac {(t+1)^2}2\rceil$ points.
Approximate spherical $t$-designs proposed in \cite{sloan2009variational} are also applied  for comparison.
Note that the minimizer of model (\ref{weighted_2}) has an explicit form (\ref{alpha_lk}) only when $t \geq 2L$, so for different $t$ we choose $L = \lfloor \frac t2\rfloor$.

\begin{figure}[!h]\label{franke_2models}
\centering
\subfigure[Uniform Errors]{ \label{franke_2models_uniform} 
\includegraphics[width=2.4in]{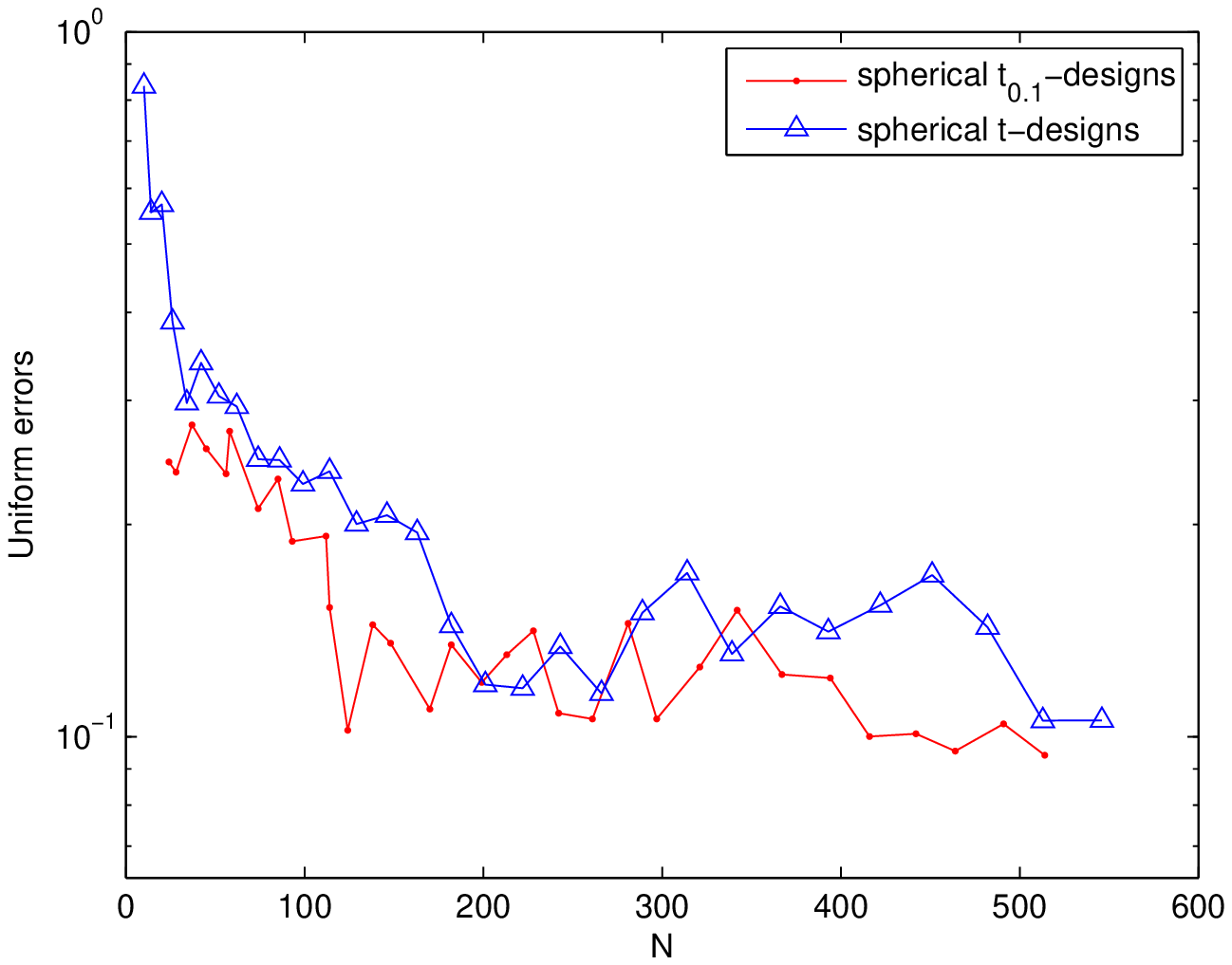}}
\subfigure[$\mathbb L_2$ errors]{ \label{franke_2models_l2} 
\includegraphics[width=2.4in]{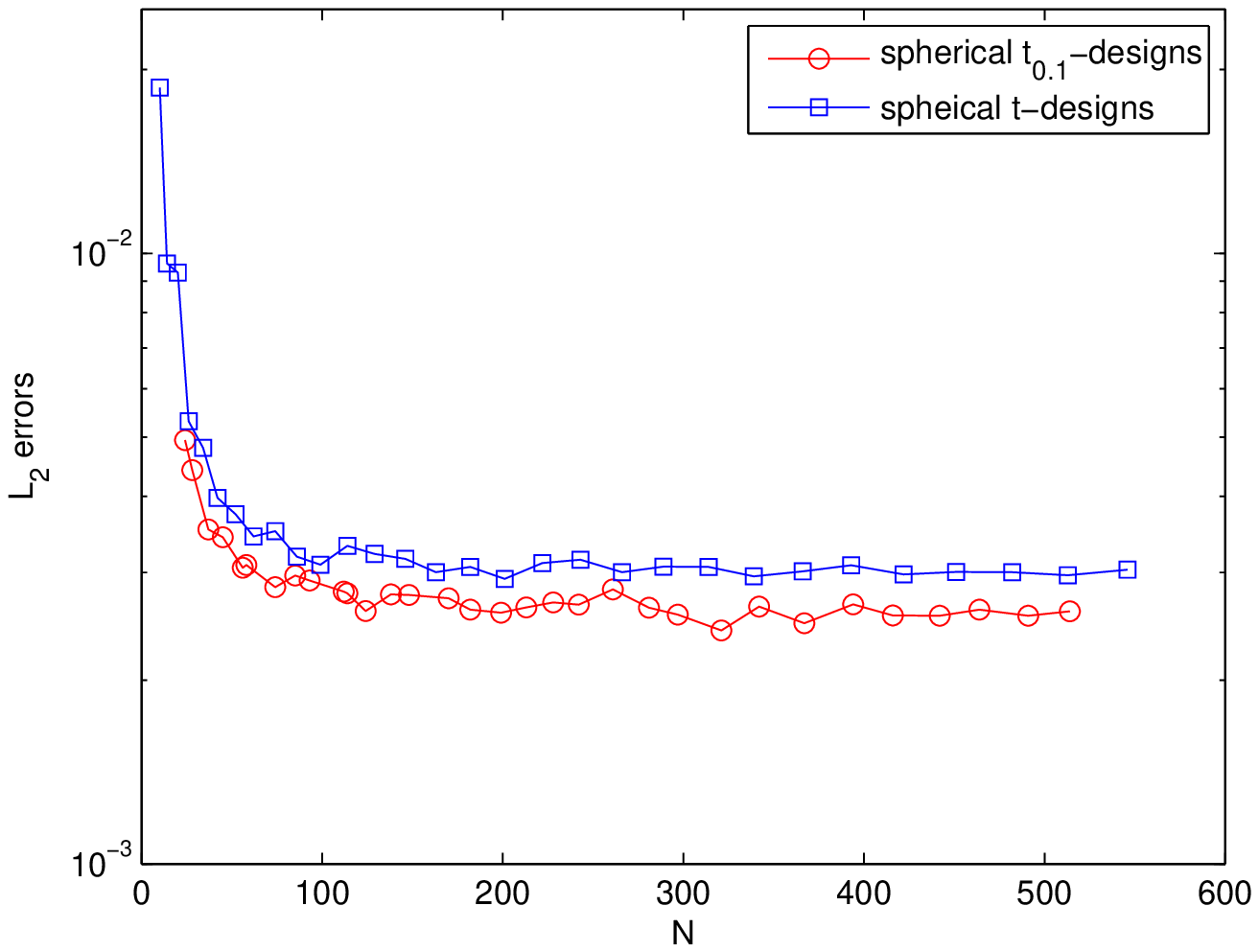}}
\caption{Errors for approximating Franke function with different scales of point sets}
\end{figure}

For Example 5.2, the approximation errors using both spherical $t_{0.1}$-designs and approximate spherical $t$-designs  are shown in Fig. 5.2.
The X-axis represents the number of points in the data sets and the Y-axis represents the minimal uniform errors.
From the figure we can see that approximation using spherical $t_{0.1}$-designs achieves smaller errors  than using approximate spherical $t$-designs in most cases.
Based on the numerical results in Fig. 5.2, the approximation quality can be improved with the relaxation of weights using model (\ref{weighted_2}).\\

\textbf{Example 5.3.} In the third experiment, a continuous but non-differentiable function
\begin{equation}\label{franke_cap}
  f_{2} = f_1(\mx) + f_{cap}(\mx),
\end{equation}
is selected as the target function to approximate, with
\begin{equation}\label{f_cap}
  f_{cap}(\mx) = \left\{
                           \begin{array}{cr}
                             \rho\cos\left(\dfrac{\pi \cos^{-1}(\mathbf{x}_c \cdot \mathbf{x})}{2r}\right), & \ \mathbf{x}\in \mathcal{C}(\mathbf{x}_c,r), \\
                             0, & \mathrm{otherwise}, \\
                           \end{array}
                         \right.
\end{equation}
where $\rho > 0$, $ 0 < r < \pi$.
The function is non-differentiable at the edge of  the  spherical cap $\mathcal{C}(\mx_c,r)$. Since the basis functions applied for approximation are spherical harmonic polynomials which is globally differentiable on $\St$, restoration of the edge of $\mathcal{C}(\mx_c,r)$ turns to be  a challenging problem when the data has noise.

\begin{figure}[!h]
 \centering
\subfigure
[ $f_2$ with $\|f_2\|_{C(\St)} \approx 3.41$]
{ \label{frankecap} 
\includegraphics[width=2.3in]{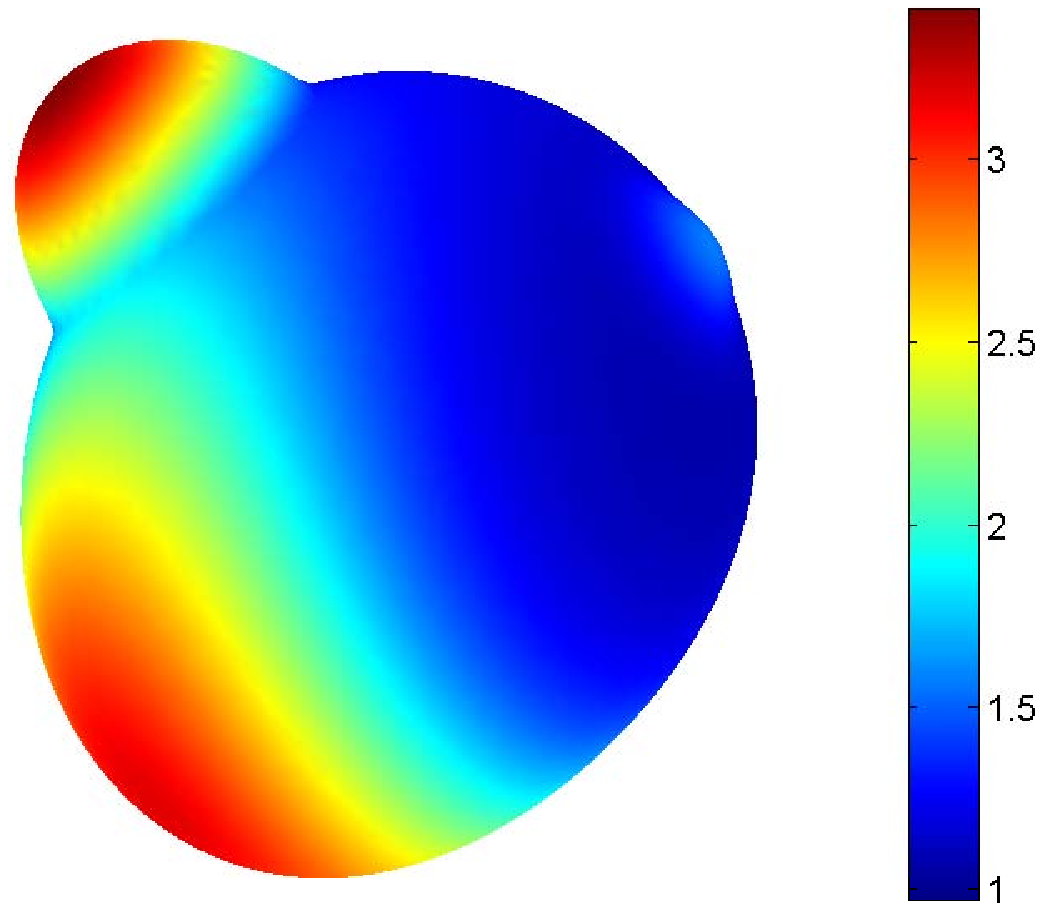}}
\subfigure[ \tiny {$f_2^\delta$ with $\delta = 0.5$}]{ \label{franke_cap_noise_5} 
\includegraphics[width=2.3in]{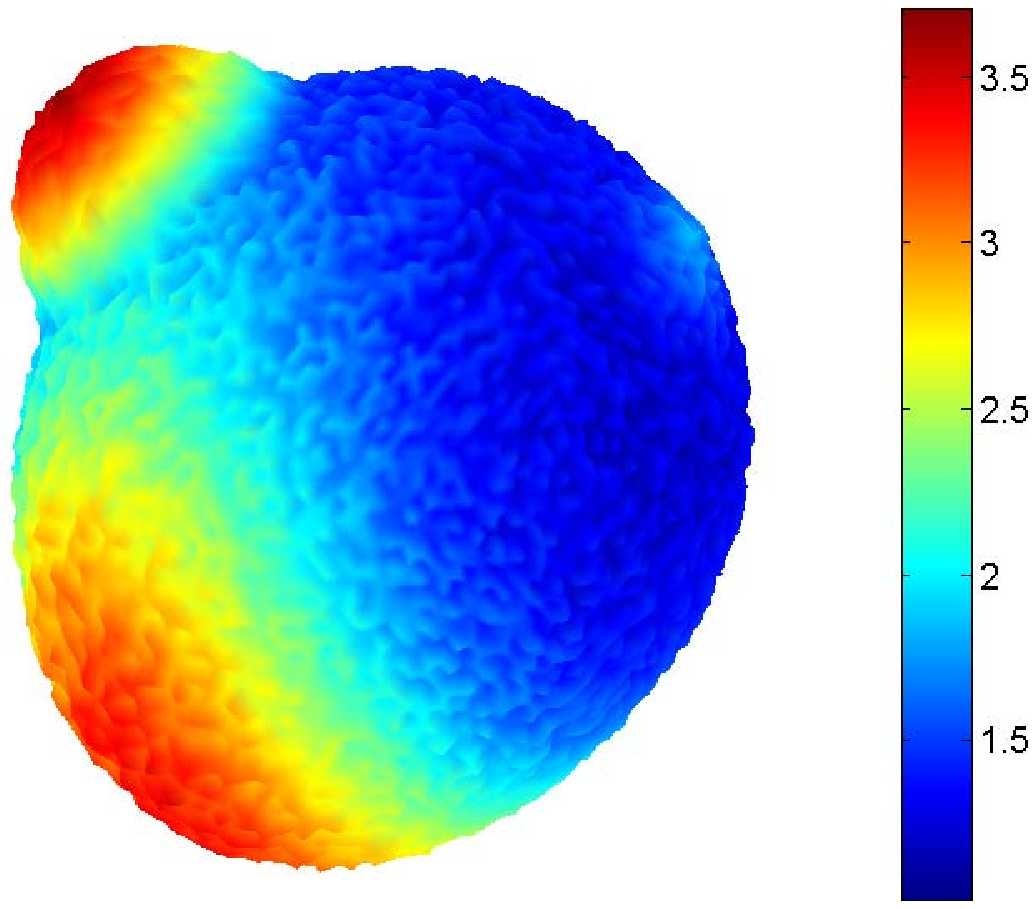}}
\subfigure[ \tiny {$l_2-l_1$ restoration with $\|p_{L,N}\|_{C(\St)} \approx 3.30$ }]{ \label{P1_frankecapnoise} 
\includegraphics[width=2.3in]{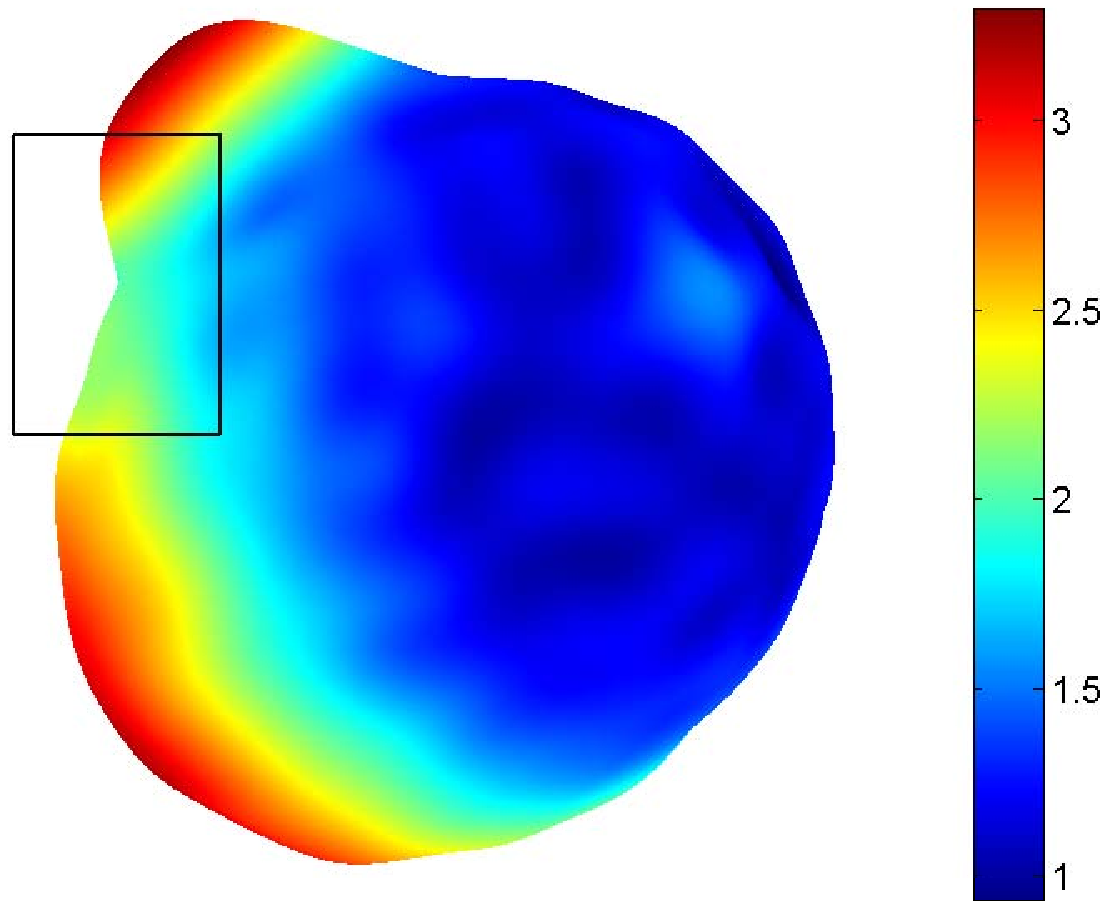}}
\subfigure[ \tiny{$l_2-l_2$ restoration with $\|p_{L,N}\|_{C(\St)} \approx 3.62$}]{ \label{P2_frankecapnoise} 
\includegraphics[width=2.3in]{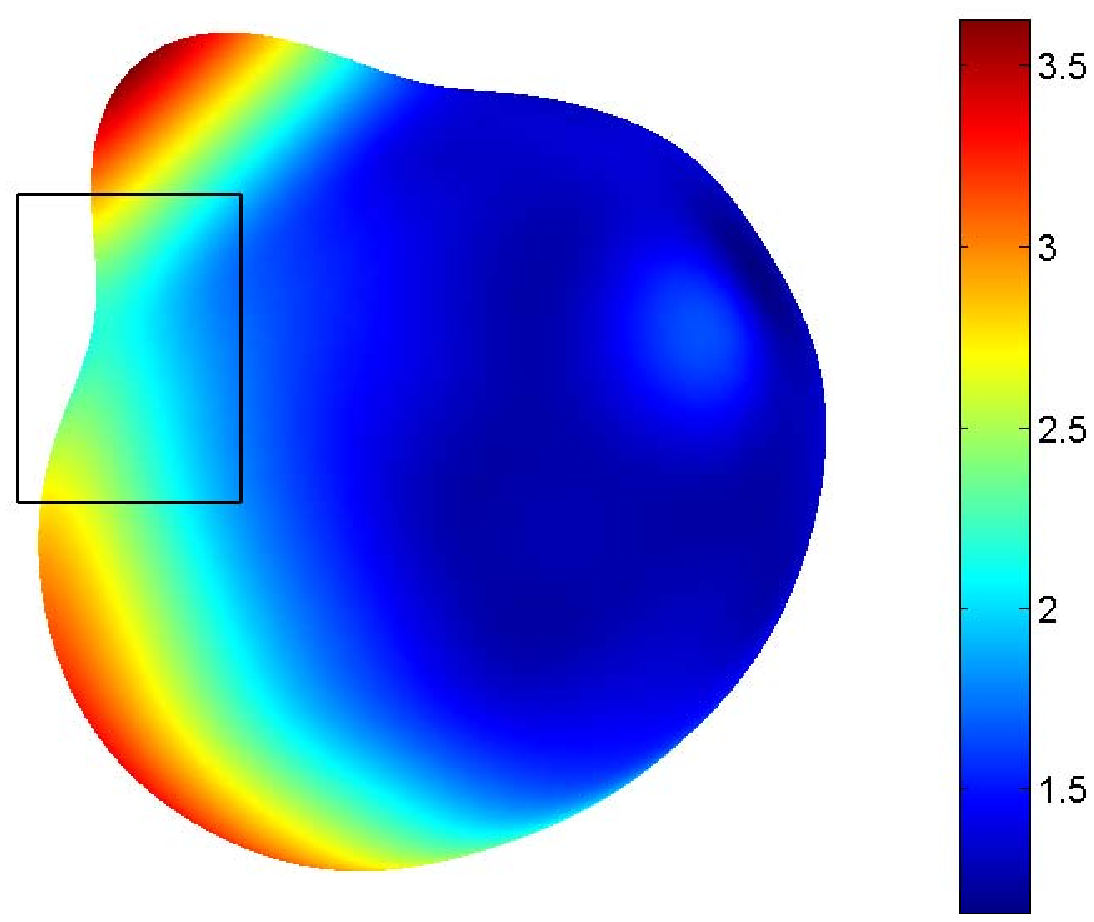}}
\subfigure[\tiny{$l_2-l_1$ restoration errors}]{ \label{errP1_frankecapnoise} 
\includegraphics[width=2.3in]{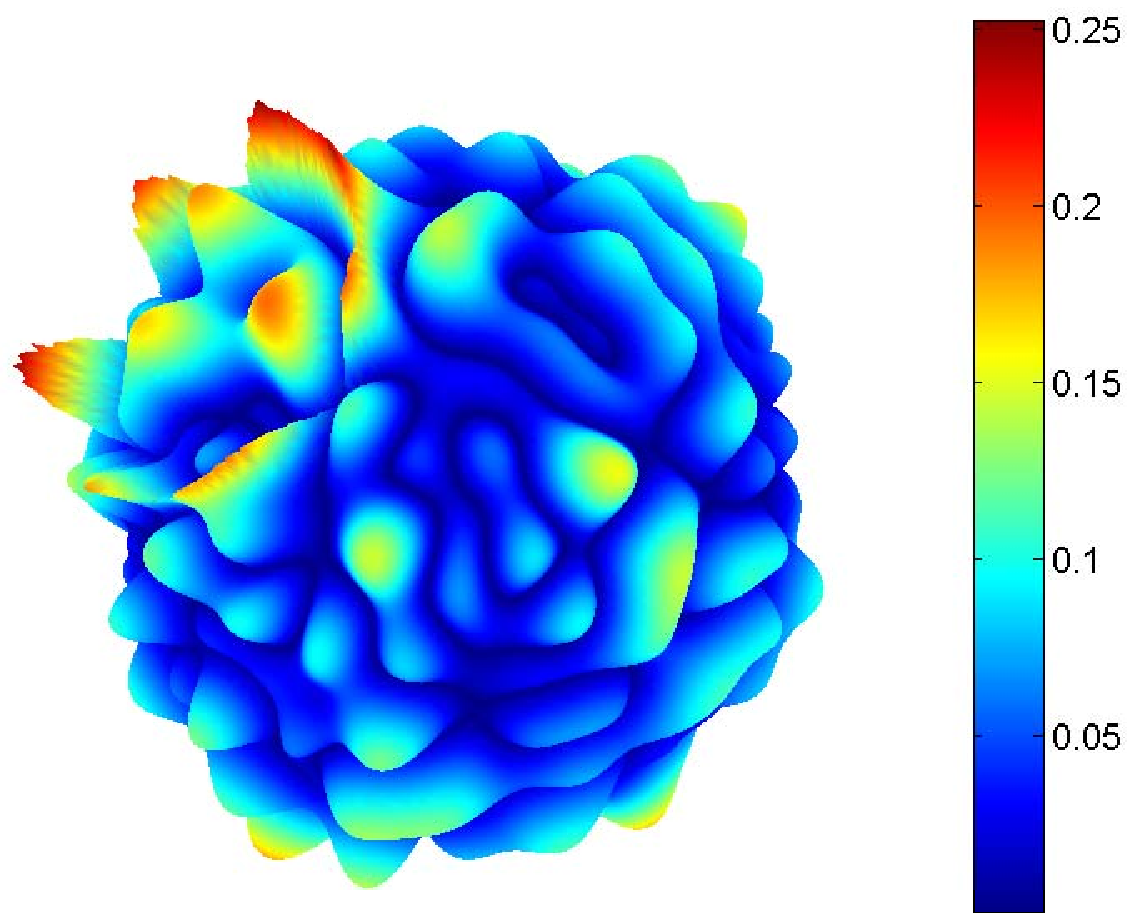}}
\subfigure[\tiny{$l_2-l_2$ restoration errors}]{ \label{errP2_frankecapnoise} 
\includegraphics[width=2.4in]{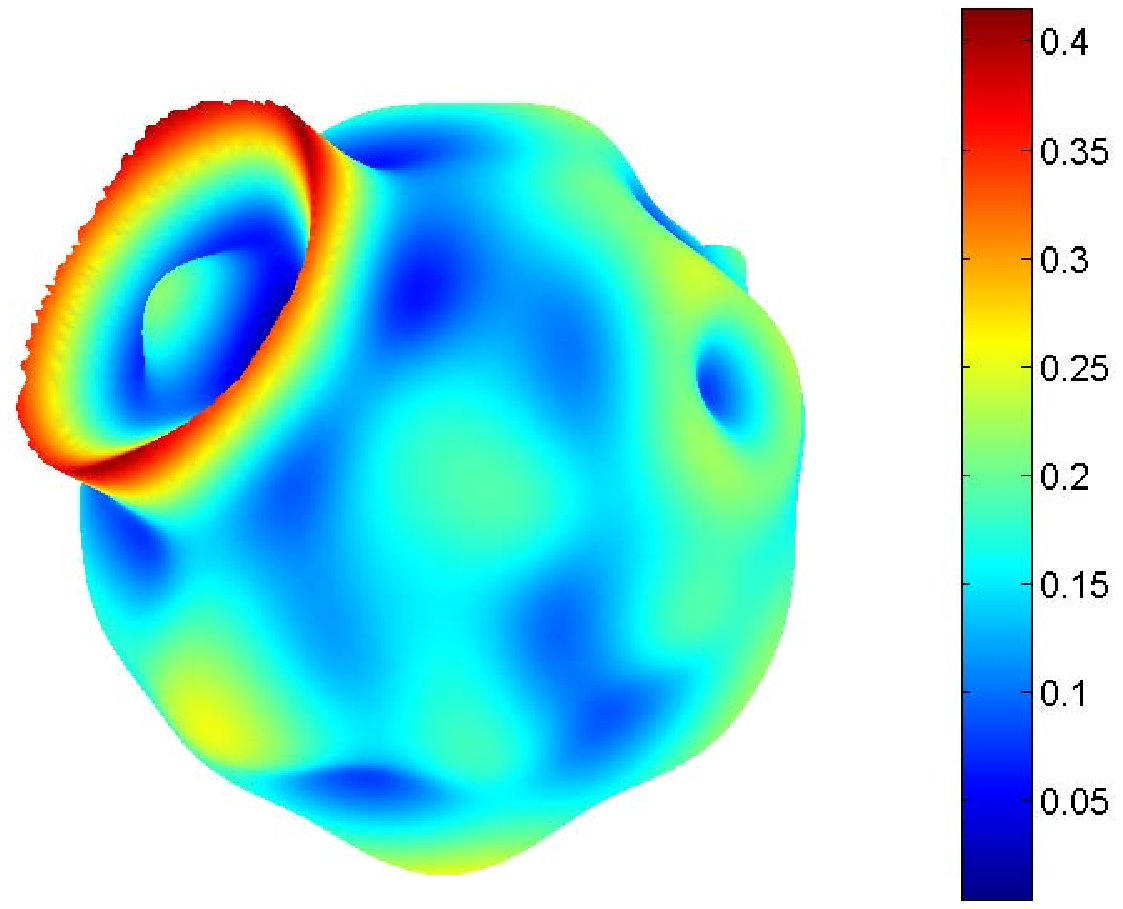}}
\caption{Restoration of $f_2$ using spherical $37_{0.1}$-design with models (\ref{weighted_2}) and (\ref{weighted_3})}
\end{figure}
A spherical $37_{0.1}$-design with 514 points is used as the data point set in this experiment. Other settings in this experiment are $\delta = 0.5$, $\lambda = 10^{-20},10^{-19.5},\ldots,10^{5}$, $\mx_c = (-0.5,-0.5,\sqrt{0.5})^T$, $r = 0.5$ and $\rho = 1$.
The restorations of $f_2$ using both models (\ref{weighted_2}) and (\ref{weighted_3}) are depicted in Fig. 5.3. Similar with previous experiments, we choose the values of $\lambda$ resulting in minimal uniform errors for each model and plot the shape of the restoration function on the sphere.
From Fig. 5.3(c)(d)(e)(f), restoration by model (\ref{weighted_2})  is not as smooth  as restoration by model (\ref{weighted_3}) but has smaller errors. And more notably, as highlighted by the rectangle in Fig. 5.3(c)(d), model (\ref{weighted_2}) restores the non-smooth edges of the spherical cap more accurately than model (\ref{weighted_3}).

\vspace{0.05in}
\noindent{\bf Acknowledgement.} We would like to thank Professor Ian Sloan  for
his valuable and helpful comments on spherical $t_\epsilon$-designs.

\end{document}